\documentclass[10pt,a4paper,final]{article}
\usepackage[utf8]{inputenc}
\usepackage[T1]{fontenc}
\usepackage{amsmath}
\usepackage{amsfonts}
\usepackage{amssymb}
\usepackage{graphicx}
\usepackage{amsthm}
\usepackage{mathtools}
\usepackage{tikz}
\usepackage[inline]{enumitem}
\usepackage{thmtools}
\usepackage{thm-restate}
\usepackage{hyperref}
\usepackage{cleveref}
\usepackage{caption}
\usepackage{subcaption}
\usepackage{cjhebrew}
\usepackage{url}
\usepackage{placeins}
\usepackage{flafter}
\usepackage{comment}
\usepackage[normalem]{ulem}
\usepackage{pdfpages}

\newtheorem{theorem}{Theorem}[section]
\newtheorem{corollary}{Corollary}[theorem]
\newtheorem{lemma}[theorem]{Lemma}

\newtheorem{conjecture}[theorem]{Conjecture}

\theoremstyle{definition}
\newtheorem{definition}{Definition}[section]

\theoremstyle{remark}
\newtheorem*{remark}{Remark}

\DeclarePairedDelimiter\abs{\lvert}{\rvert}%
\DeclarePairedDelimiter\norm{\lVert}{\rVert}%

\newcommand{\interior}[1]{%
	{\kern0pt#1}^{\mathrm{o}}%
}

\newcommand{\R}{\mathbb{R}}
\newcommand{\C}{\mathbb{C}}
\newcommand{\NN}{\mathbb{N}}

\newcommand{\M}{\mathcal{M}}
\newcommand{\N}{\mathcal{N}}
\newcommand{\OO}{\mathcal{O}}
\newcommand{\B}{\mathcal{B}}

\newcommand{\A}{\Lambda}
\newcommand{\LL}{\mathcal{L}}
\newcommand*\colvec[2]{\begin{pmatrix} #1 \\ #2\end{pmatrix}}

\newcommand{\mdiags}[1]{#1 \setminus \text{Diag}((-1,1))}

\title{Almost Regular Closedness of the Connectedness Locus for Pairs of Affine Maps on $\mathbb{R}^2$}
\author{Omer Rosler}

\begin{document}
	
\maketitle

\begin{abstract}

We study the connectedness locus $\N$ for the family of iterated function systems of pairs of homogeneous affine-linear maps in the plane. We prove this set is regular closed (i.e., it is the closure of its interior) away from the diagonal, except possibly for isolated points, which we conjecture do not exist. We provide an overview of the "method of traps", introduced by Calegari et al.\ in \cite{calegari}, which lies at the heart of our proof.
	
\end{abstract}
\tableofcontents

\section{Introduction}\label{sec:intro}
Let $X$ be a complete metric space, and suppose $f, g: X\to X$ are two contractions. $\{f,g\}$ is called an iterated function system (IFS henceforth). There exists a unique non-empty compact set $\A\subset X$, called the {\it attractor}, that satisfies
\begin{equation}\label{def:atractor_property}
	\A = f(\A)\cup g(\A).
\end{equation}
It is well known (see \cite{Hata1985OnTS}) that if $f$ and $g$ are injective, then there is a topological dichotomy concerning the connectedness properties of $\A$:
That is, if $f\A\cap g\A \neq \emptyset$, then $\A$ is connected, otherwise $\A$ is a topological Cantor set.

If $f,g$ depend upon a parameter from some metric space $P$ in a continuous matter, that is, if we have maps $f,g:P\times X \to X$ such that $\{f(p,\cdot), g(p,\cdot)\}$ forms an IFS for all $p\in P$, then the map $\A: P\to H(X)$ defined by the attractor of the IFS for a given parameter, is continuous (where $H(X)$ denotes the space of compact subsets of $X$ equipped with the Hausdorff metric). For a proof, see \cite[Theorem~3.4]{jachymski1996continuous}. Usually we denote the parameter by a subscript, omitting it when  clear from the context. 

We define the connectedness locus associated with the family of IFSs:

\begin{definition}
	$\mathcal{L}(f,g)=\{p\in P : \A_p \text{ is connected}\}$
\end{definition}

The continuity of the attractor motivates the study of connectedness loci as topological spaces, which we will focus on. More specifically, finding interior points and characterizing them.

Connectedness loci exhibit fractal behavior, with a wide variety of geometric phenomena. Hence, in order to prove meaningful results, we need to specialize further and add more assumptions about the spaces and maps involved.

Arguably, the simplest cases to consider are homogeneous self-affine IFSs on a euclidean space. This means that both maps are affine, and differ by a constant vector. Equivalently, there exists a matrix $T\in\M_{d\times d}(\R)$ and a pair of vectors $b_1,b_2\in\R^d$ such that $f(x)=Tx+b_1,\ g(x)=Tx+b_2$ (where $X=\R^d$). Let $b=b_2-b_1$. Note that by applying a change of coordinates, we may assume $b_1=0$. Without loss of generality, we may also assume that $b$ is a cyclic vector, that is,
$$
Y:={\text{span}\{T^k b:k\geq 0\}}=\R^d.
$$ 
Otherwise, we replace $X$ with $Y$ and consider $T$ as a mapping on $Y$. 
Write
$$
\B = \big\{1+\sum_{n=1}^{\infty}a_nx^n: a_n\in\{-1,0,1\}\big\},
$$
and note that $\B$ is compact in the topology of uniform convergence (Lemma \ref{thm:B_compactness}).

The homogeneous self-affine case is relatively simple due to two reasons: 

First, the attractor has an alternate description using power series which allows us to apply tools from analysis:

\begin{lemma}\label{thm:att_linear}
	If $b$ is cyclic then the attractor of $\{Tx,Tx+b\}$ is given by power series in the following sense:
	$$
	\A = \left\{\sum_{n=0}^{\infty }a_n T^nb: a_n\in\{0,1\}\right\}
	$$
\end{lemma}
\begin{proof}
	The right-hand side is non-empty, compact and satisfies (\ref{def:atractor_property}), therefore this claim follows from the uniqueness of the attractor.
\end{proof}

This description allows us to characterize connectedness using power series as well:
\begin{lemma}\label{thm:locus_based_on_eigen}
	Suppose $T$ is a linear contraction with eigenvalues $\lambda_j$ for $j=1,\dots,m$ having algebraic multiplicities $k_j\geq 1$ and geometric multiplicities equal to one. Let $b$ be a cyclic vector for $T$. Then the attractor of $\{Tx,Tx+b\}$ is connected iff there exists an $f\in\B$ such that
	\begin{equation} \label{f_characterization}
		f(\lambda_j)=\cdots=f^{(k_j-1)}(\lambda_j)=0, j=1,\dots,m
	\end{equation}
	
	In particular, connectedness does not depend on $b$. 
\end{lemma}
\begin{proof}
	See \cite[Proposition~2.4]{shmerkin_sol2000} with the proof in Section 7.
\end{proof}

Second, there is a very simple criteria that implies a point is in the locus, using a volume argument:
\begin{lemma}[Folklore]\label{thm:folklore}
	Suppose $T_1,T_2:\R^d\to \R^d$ are linear maps and $b_1,b_2\in \R^d$ are two vectors.
		If $\{T_1x+b_1,T_2x+b_2\}$ is an IFS of contracting maps, that is $\max\{\norm{T_1},\norm{T_2}\}<1$ for some operator norm, and $\abs{\det(T_1)}+\abs{\det(T_2)}\geq 1$, then the attractor is connected.
\end{lemma}
\begin{proof}
	See \cite[Lemma~2.3]{shmerkin_sol2000} with the proof in Section 7.
\end{proof}

Combining Lemmas \ref{thm:locus_based_on_eigen} and \ref{thm:folklore} yields the following characterization:

\begin{corollary}\label{thm:general_trivial}
	If $\{\lambda_1,\dots,\lambda_m\}$ are the eigenvalues of $T$ and they all satisfy $\abs{\lambda_j}<1$, have algebraic multiplicity $k_j\geq 1$ and also
	\begin{equation} \label{the_trivial_ondition}
		\prod_{j=1}^{m}\abs{\lambda_j}^{k_j}\geq \frac{1}{2}.
	\end{equation}
	Then there exists $f\in\B$ having zeros at $\lambda_j$ of algebraic multiplicity $\geq k_j$, for ${j=1,\dots,m}$. In accordance with Lemma \ref{thm:folklore}, this implies that the attractor is connected.
\end{corollary}

Therefore, the natural way to choose families of homogeneous self-affine IFSs is to consider the "type" of  eigenvalue multiplicities separately. We will restrict our focus for maps on the plane, in which case there are only three such "types", each corresponds to a separate set of parameters and connectedness locus.
After applying an invertible linear transformation, we may assume, without loss of generality that $T$ is one of the following:

\begin{enumerate*}
	\item $T=\begin{pmatrix}
		a & b \\
		-b & a
	\end{pmatrix}$,
	\item $T=\begin{pmatrix}
		\gamma & 0 \\
		0 & \lambda
	\end{pmatrix}$,
	\item $T=\begin{pmatrix}
		\lambda & 1 \\
		0 & \lambda
	\end{pmatrix}$,
\end{enumerate*}

where $\gamma,\lambda,a,b$ are real, $\abs{\gamma},\abs{\lambda}<1$ and $a^2+b^2<1$. Note that $\gamma\neq\lambda$ by the assumption that $T$ has a cyclic vector.
Each of these cases leads to a separate connectedness locus for the corresponding family of self-affine sets. Namely, we consider the sets
\begin{eqnarray*}
& \M := & \{z=a+ib: a^2+b^2<1 \text{ and there exists } f\in \B,\;  f(z)=0\} \\
& \N := & \{(\gamma,\lambda)\in(-1,1)^2: \text{ there exists } f\in \B,\; f(\gamma)=f(\lambda)=0\} \\
& \OO := & \{\lambda\in(-1,1): \text{ there exists } f\in \B ,\; f(\lambda)=f'(\lambda)=0\}
\end{eqnarray*}
We want the loci to be closed, therefore we must include certain degenerate subsets: for $\M$ we include real axis $\R\cap(-1,1)$ and for $\N$ we include the diagonal, henceforth denoted by
$$
\text{Diag}((-1,1))=\{(\lambda,\lambda)\in\R^2:\lambda\in (-1,1)\}
$$

Because the degenerate subsets are different topologically, we will exclude them when analyzing the topological properties of the loci.

High-dimensional analogs are also possible (based on the "eigenvalues pattern" of the linear maps).

\begin{figure}[h]
	\centering
	\fbox{\includegraphics[scale=0.8]{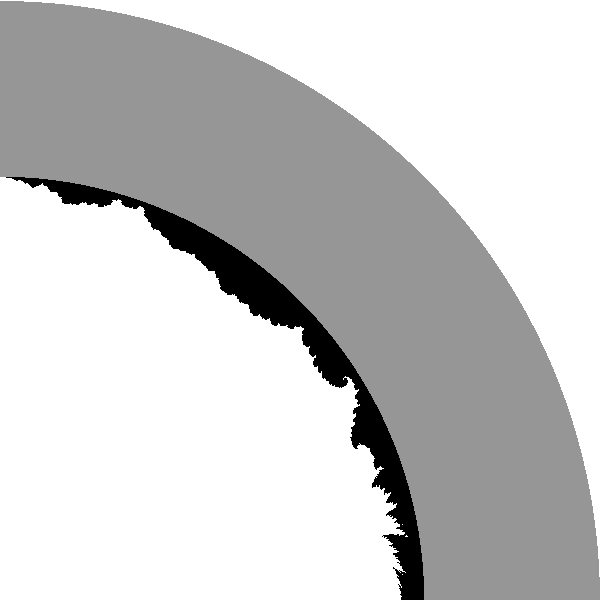}}
	\caption{$\M\cap[0,1]^2$ in black and gray. The gray part indicates the trivial part.}
	\label{fig:M}
\end{figure}

\begin{figure}[h]
	\centering
	\begin{subfigure}[h]{0.45\textwidth}
		\centering
		\fbox{\includegraphics[width=\textwidth]{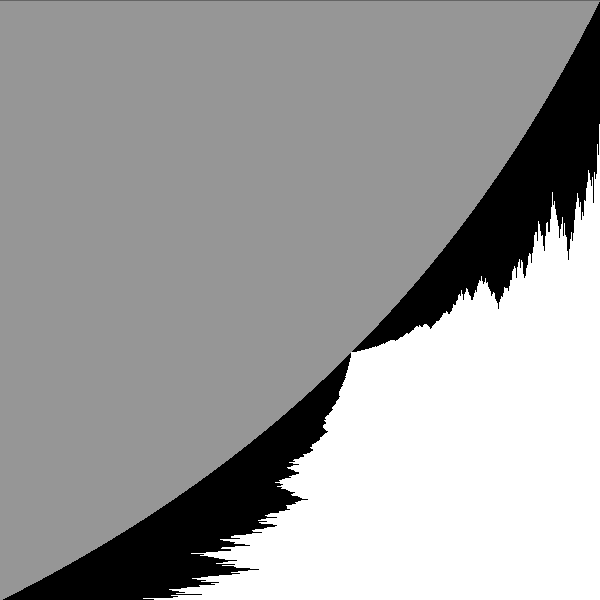}}
		\caption{$\N\cap([-1,-\frac{1}{2}]\times[\frac{1}{2},1])$}
		\label{fig:Nm}
	\end{subfigure}
	\hfill
	\begin{subfigure}[h]{0.45\textwidth}
		\centering
		\fbox{\includegraphics[width=\textwidth]{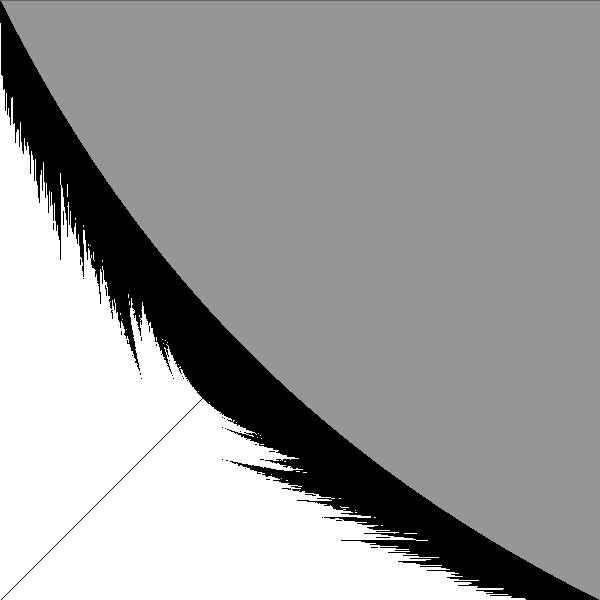}}
		\caption{$\N\cap[\frac{1}{2},1]^2$}
		\label{fig:Np}
	\end{subfigure}
	
	\caption{The set $\N$ in black and gray. The gray part indicates the trivial part.}
	\label{fig:N_plots}
\end{figure}

There is no figure of $\OO$ here, we refer to \cite{shmerkin_sol2000} for the details of $\OO$.

$\M$  was coined by Bandt as "the Mandelbrot set for pairs of linear maps" in \cite{bandt_orig} and studied as an analog of Mandelbrot set. It has a intricate structure and we refer to \cite{bandt_orig} for an in-depth analysis of the different geometric behaviors.

$\N$ is more straight-forward, and by zooming in we reveal that the behavior is mostly of "spikes" as we can see in Figure \ref{fig:N_zoom_spikes}. Around the diagonal it looks simpler as apparent in Figure \ref{fig:N_diag_fig}.
For a detailed discussion on the geometry of $\N$ we refer to \cite{N_general}.

\begin{figure}[h]
	\centering
	\begin{subfigure}[h]{0.45\textwidth}
		\centering
		\fbox{\includegraphics[width=\textwidth]{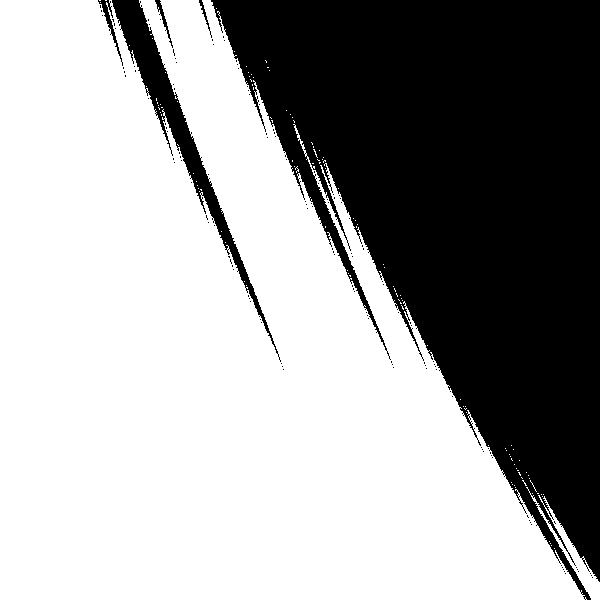}}
		\caption{$\N\cap[0.647,0.661]\times[0.677,0.691]$}
		\label{fig:N_zoom_spikes}
	\end{subfigure}
	\hfill
	\begin{subfigure}[h]{0.45\textwidth}
		\centering
		\fbox{\includegraphics[width=\textwidth]{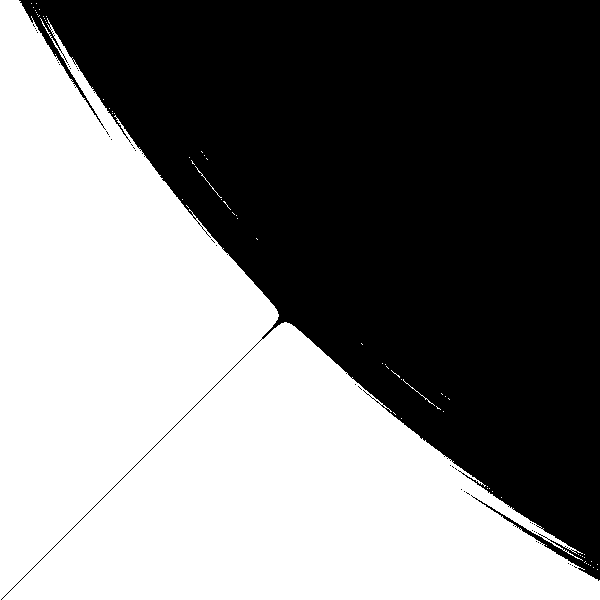}}
		\caption{$\N\cap[0.658,0.8]^2$}
		\label{fig:N_diag_fig}
	\end{subfigure}
	
	\caption{Zoomed in parts of $\N$}
	\label{fig:N_zoom}
\end{figure}
\begin{remark}
	All the figures describing $\M$ and $\N$ were generated by an escape-time algorithm described in \cite[Section~9]{bandt_orig} (with slight modifications to accommodate $\N$).
	Due to the fractal nature of these sets there are artifacts in the images.
	In Figure \ref{fig:N_zoom_spikes} there are small "islands", and in Figure \ref{fig:N_diag_fig} there is smoothness around the diagonal. These features are not real.
\end{remark}

Although $\M$ and $\N$ have similar defining properties (and similar topology), their geometry is widely different (see Figures \ref{fig:M}, \ref{fig:N_plots} and \ref{fig:N_zoom}). Therefore, finding shared properties of $\M$ and $\N$ is important, as it sheds light on the nature of the topology of connectedness loci in general. Bousch proved $\M$ is locally connected in \cite{Bousch1992SurQP} and \cite{bousch1993connexite}, while for $\N$ only a partial results exists due to Solomyak \cite{N_general}. We focus on one conjectured shared property -- being regular closed (that is, it is the closure of its interior) except for a specific degenerate subset. For $\M$ (minus the real axis), it was originally conjectured by Bandt in \cite{bandt_orig} and proven by Calegari et al.\ in \cite{calegari}. We adapt their method of proof to $\N$ and prove "almost" regular closedness: $\mdiags{\N}$ is regular closed except perhaps for isolated points.

The simplest way to find interior points for affine maps is using Lemma \ref{thm:folklore}, because the inequality is stable under perturbations of the parameter. This subset of parameter space is called "the trivial part". For $\M,\N,\OO$, these subsets are
\begin{eqnarray*}
	& \M_t := & \{z\in \C : 2^{-1/2}\leq\abs{z}<1\}, \\
	& \N_t := & \{(\gamma,\lambda)\in(-1,1)^2 : \frac{1}{2}\leq\abs{\gamma\lambda}<1\}, \\
	& \OO_t := & \{\lambda\in(-1,1) : 2^{-1/2}\leq\abs{\lambda}<1\}.
\end{eqnarray*}

Finding interior points which are not trivial is generally hard, even for linear IFSs. For example, in \cite[Theorem~2.11]{shmerkin_sol2000}, explicit interior points were found. Recently, a similar method was used in \cite{espigule2024collinear} to prove regular closedness of related connectedness loci. The "method of traps" introduced by \cite{calegari} allows finding interior points by checking simple geometric conditions of the attractor. Their method was described only for $\M$, and we adapt it to $\N$. Other adaptations exists, for example in \cite{himeki2020regular}, they used a modified version to prove regular closedness of a related connectedness locus.

The paper is structured as follows:
Section \ref{sec:results} summarizes the main results of the paper. Section \ref{sec:notation} introduces the notation we use. Section \ref{sec:traps} provides an overview of the method of traps introduced by Calegari et al.\ \cite{calegari}. In Section \ref{sec:N}, we specialize to the case of $\N$ and obtain a partial result regarding regular closedness. In Section \ref{sec:analytic} we extend this result using analytical methods and prove regular closedness (away from the diagonal) except for possibly isolated points.

\section{Summary of the Results}\label{sec:results}

The primary application of traps in \cite{calegari} is to prove regular closedness of the locus. This reduces to proving the existence of traps arbitrarily close to a given parameter. We leverage more structure of the locus, to prove a given vector is a translation vector between some cylinder sets. In \cite{calegari} this is accomplished using the Surjective Perturbation Lemma.

In the case of $\N$, this lemma generally fails; however, by characterizing when it holds (Lemma \ref{thm:surjective_perturbation_N}) we identify a subset of $\N$ that is regular closed. This set, denoted $\N'_U$ consists of pairs of zeros of a power series $f\in\B$, for which these zeros are both sign-changes of $f$, and satisfy some other regularity condition (Definition \ref{def:N_u}). We prove the analog of \cite[Theorem~7.2.7]{calegari} in almost the same exact manner. This is the following theorem:

\begin{restatable*}[Traps are dense in $\N'_U$]{thm}{densetraps}
	\label{thm:calegari_analog}
	$\mdiags{\N'_U}\subseteq \overline{\interior{\N}}$
\end{restatable*}

In Section \ref{sec:analytic}, we push the result obtained by traps and generalize it. We do this by finding which points of $\N$ are not limits of interior of points (using the already established interior points from Theorem \ref{thm:calegari_analog}). We denote the set of such "outlier points" by $\N''$ (Definition \ref{def:outlier_points}). We impose restrictions on these points using analytic methods. We prove that if $(\gamma,\lambda)\in\N$, with $f\in\B$ such that $f(\gamma)=f(\lambda)=0$, then $f$ is of very particular form:

First, we restrict the order of zeros (Lemma \ref{thm:eliminate_case}), then we determine the tail of $f$ (Lemma \ref{thm:tail_shape}). This shape of the tail implies these points are algebraic:

\begin{restatable*}{corollary}{algebraic}\label{thm:outlier_algebraic}
	Every $(\gamma,\lambda)\in \N''$ is algebraic.
\end{restatable*}

Finally, we establish our strongest partial restriction on outlier points:

\begin{restatable*}{theorem}{isolated}\label{thm:outlier_isolated}
	Every $(\gamma,\lambda)\in \N''$ is isolated in $\N$.
\end{restatable*}

We conjecture that there are no outlier points.

\section{Notation and Preliminaries}\label{sec:notation}
\subsection{General IFS}
Let $\{f,g\}$ be an IFS on a euclidean space $X=\R^d$.
We define the norm of $f$ using an operator norm. In particular if  $f(x)=Tx+b$ then $\norm{f}:=\norm{T}$. We define the norm of the IFS $\{f,g\}$ as the supremum of the norms of $f$ and $g$, that is, $\norm{\{f,g\}}=\sup\{\norm{f},\norm{g}\}$.

We denote the symbolic space of $k$-letters by $\Sigma^\infty_k$, which consists of right-infinite words formed from the letters $1,\dots k$. We equip this space with the usual prefix metric, defined by $d(u,v)=2^{-\abs{u\land v}}$, where $u\land v$ is the common prefix of the two words, In this paper, the term "word" refers to an element of this set unless specified as "finite".
For an IFS of maps $f_1,\dots,f_k$, we identify the alphabet with these maps. 

For a finite word $u^m$ we usually denote the length as a superscript (and omit it where clear from the context).

When the alphabet is from an IFS, we denote by $u^m\A$ the cylinder set associated with $u^m$. Such finite word induces a map on $X$ as well; by abuse of notation we denote this map by $u^m$ as well.

For both a finite and infinite word $w$, the subscript $w_i$ indicates the $i$-th letter.

Concatenation of a finite word $w$ and a (possibly finite) word $u$ is denoted by $wu$.

For an infinite word $w$, the truncation to the first $m$ letters is denoted by $w^m$.

The address map is denoted $\pi:\Sigma^\infty_k\times P \to X$, it maps a word $w$ and a parameter $\lambda$ to a unique point $\pi(w, \lambda)\in\A_\lambda$ which is defined by
$$
\{\pi(w,\lambda)\}=\bigcap_{m=1}^\infty w^m\A_\lambda
$$ (this set is a singleton due to Cantor's lemma). We omit the parameter where it is clear from the context, and simply denote $\pi(w)$.

\subsection{Choice of Normalization for $\M$ and $\N$}
When referring to $\M$, we will use the family of IFSs on $\mathbb{C}$ (which is equivalent to the linear description above):
$$
f(x)=zx, \ g(x)=zx+(1-z),
$$
where $z$ is a complex parameter.

When referring to $\N$, we will use the family of IFSs on $\R^2$:
$$
p(x)=Tx+s_p, m(x)=Tx-s_p
$$
where $p$ stands for "plus", $m$ for "minus" and $s_p=\colvec{1}{1}$. This ensures the center of symmetry of $\A$ is the origin. Additionally, this aligns the notation with \cite{hare_sidorov_2017}. See Figure \ref{fig:att_plot} for some examples of the attractor.

When we refer to any connectedness locus, we denote it by $\mathcal{L}$.

There are several symmetries in these sets which we use implicitly throughout the paper:
$$\lambda\in\LL \iff -\lambda\in\LL$$ for all three loci,
$$z\in\M\iff \bar{z}\in\M,$$
and 
$$(\gamma,\lambda)\in \N \iff (\lambda,\gamma)\in \N.
$$

\begin{figure}[h]
	\centering
	\begin{subfigure}[b]{0.3\textwidth}
		\centering
		\fbox{\includegraphics[width=\textwidth]{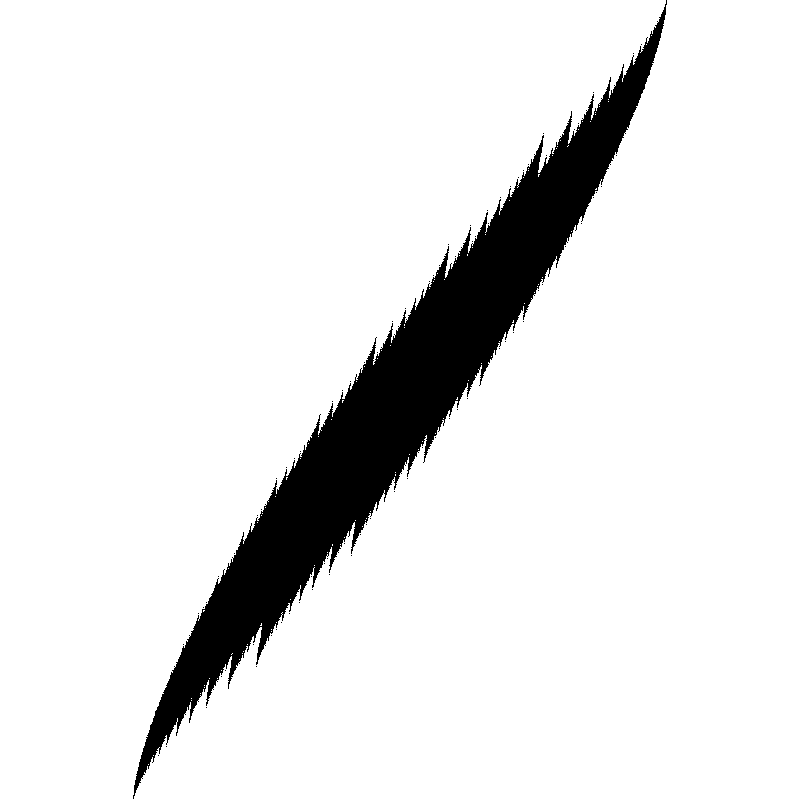}}
		\caption{$\gamma=\frac{1}{1.2},\lambda=\frac{1}{1.3}$}
	\end{subfigure}
	\hfill
	\begin{subfigure}[b]{0.3\textwidth}
		\centering
		\fbox{\includegraphics[width=\textwidth]{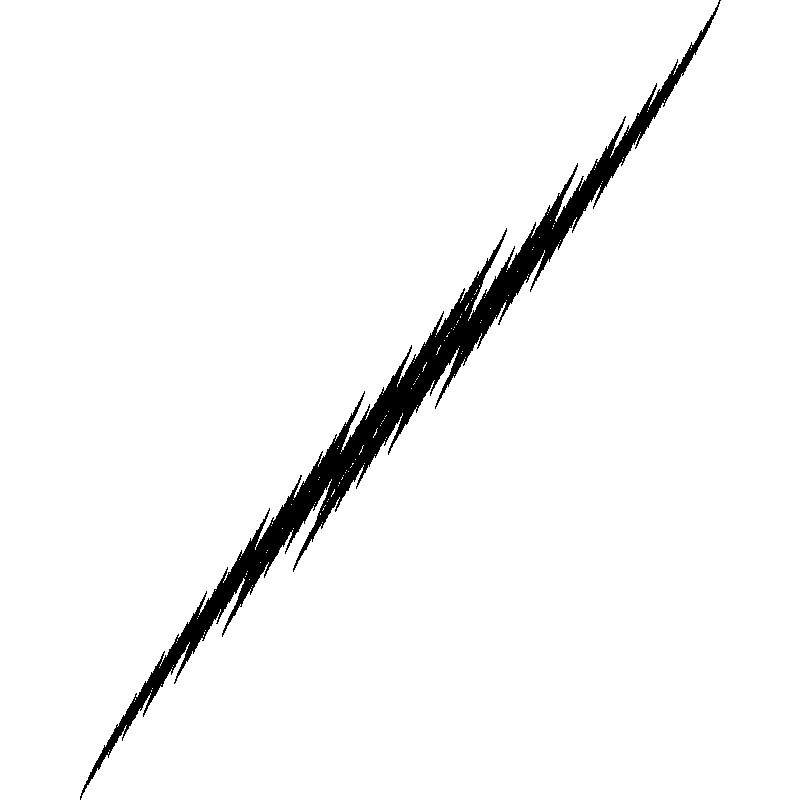}}
		\caption{$\gamma=\frac{1}{1.4},\lambda=\frac{1}{1.5}$}
	\end{subfigure}
	\hfill
	\begin{subfigure}[b]{0.3\textwidth}
		\centering
		\fbox{\includegraphics[width=\textwidth]{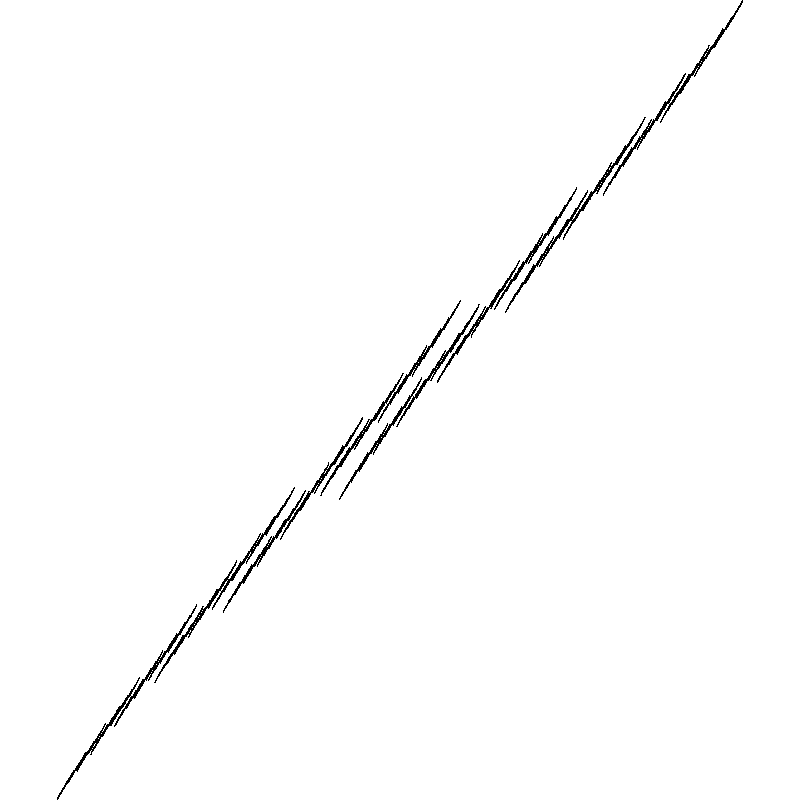}}
		\caption{$\gamma=\frac{1}{1.6},\lambda=\frac{1}{1.7}$}
	\end{subfigure}
	
	\caption{The attractor for the linear family in $\N$. The extremal points are $\pm(\frac{1}{1-\gamma},\frac{1}{1-\lambda})$}
	\label{fig:att_plot}
\end{figure}

\begin{remark}
	A priori, for $\N$ we also need to remove the "anti-diagonal", that is,
	$$
	\{(\lambda,-\lambda):\lambda\in(-1,1)\}.
	$$
	However, we claim all such points are in the trivial part of $\N$. Indeed, let $f\in\B$ such that $f(\lambda)=f(-\lambda)=0$. Write $f(x)=\sum_{n=0}^\infty a_nx^n$ where $a_n\in\{-1,0,1\}$. Then we can add the equations and get $0=f(\lambda)+f(-\lambda)=2\sum a_{2n}\lambda^{2n}$ and therefore $\lambda^2$ is a zero of a power series from $\B$. This implies $\lambda^2>\frac{1}{2}$ and hence $(\lambda,-\lambda)\in\N_t$.
\end{remark}

\section{Overview of the method of traps}\label{sec:traps}
This section provides an overview of the method of traps introduced in \cite{calegari}, where we slightly weaken the assumptions for our use case.

The method of traps is a simple yet useful tool to find interior points of connectedness loci in the plane. The goal is to force an intersection of two cylinder sets, in a way that is preserved under any perturbation of the parameter.

The core idea involves finding intersecting paths which are described by points in the symbolic space, and the proof of them being intersecting does not explicitly depend on the parameter, only on geometric properties of the attractor. Thus, the intersection is preserved under perturbation of the parameter (due to continuity of the address map). By choosing the paths correctly, one can prove that the point is in the interior of the connectedness locus  (Theorem \ref{thm:trap_implies_interior}). The construction of the paths in \cite{calegari} and in our paper relies on the IFS maps being affine (most notable in Lemma \ref{thm:short_hop_stability}), therefore we will focus on this case.

More explicitly, the construction utilizes "short-hop paths" (Definition \ref{def:short_hop_path}) which are  piecewise linear paths in the attractor, with finitely many interim vertices that satisfy certain local distance requirements. It turns out nothing else is required due to the "self-similar" (strictly speaking, self-affine) structure of the attractor.

To find such paths that are guaranteed to intersect, we need more assumptions. When assuming the IFS is self-affine and homogeneous, there is a clever way to show this (introduced in \cite{calegari}). We find cylinder sets that are "intertwined" (Definition \ref{def:trap_like}). In the homogeneous case, any two cylinder sets of the same order are translations of each another. Therefore, the condition of them being "intertwined" reduces to a simple geometric property of the translation vectors between the cylinder sets (Theorem \ref{thm:trap_in_affine_case}).

The final ingredient for constructing a trap is to find a normalized translation vector that satisfies this geometric condition. For every application this is done differently.

The usefulness of traps lies in the {\it simplicity} of verifying this condition, which implies the point is in the interior of the locus.

Throughout this section, assume $X$ is a plane, i.e.  $X=\R^2\text{ or }\C$. Also assume the IFS is self-affine, that is, $f,g$ are affine maps for all parameters.
\subsection{Short-hop Paths}
First, we demonstrate the existence of special paths between points in the same cylinder set. The paths used by \cite{calegari} are defined as follows:

\begin{definition}[Short-hop path, {\cite[Definition~7.1.1]{calegari}}]\label{def:short_hop_path}
	Let $p,q\in \A$, $\epsilon>0$ and $D$ be a topological disk containing $p,q$. An $(\epsilon, D)$-short-hop path from $p$ to $q$ is a finite sequence of words $e_0,\dots,e_m$ where $\pi(e_0)=p, \pi(e_m)=q$ and $d(\pi(e_i),\pi(e_{i+1}))<\epsilon$ and $\pi(e_i)\in D$ for all $i$.
\end{definition}
Some useful properties are given by the following lemma, which is similar to \cite[Proposition~7.1.2]{calegari} :
\begin{lemma}\label{thm:short_hop_properties}
\begin{enumerate}
	\item Let $u^n$ be a finite word of length $n$ (we  omit $n$). Suppose $e_0,\dots,e_m$ is an $(\epsilon, D)$-short-hop path. Then $uD$ is a topological disk and $$ue_0,\dots,ue_m$$ is an $(\norm{\{f,g\}}^n\epsilon,uD)$ short-hop path.
	\item Let $h$ and $h'$ be short-hop paths that intersect transversely. Then $uh,uh'$ also intersect transversely (where $uh$ refers to the short-hop path from the first part).
\end{enumerate}
\end{lemma}
\begin{proof}
	This is trivial, noting $f,g$ are affine and hence so is $u$.
\end{proof}
\begin{lemma}\label{thm:short_hop_stability}
Let $\delta=d(f\A,g\A)$.
\begin{enumerate}
\item The set $N_\frac{\delta}{2}(\A)$ is path-connected; indeed, there exists short-hop paths between any two points in $\A$, lying entirely inside this neighborhood.
\item Suppose $p,q \in u^n\A$ where $u^n$ is a finite word of length $n$. Then there exists an $\frac{\norm{\{f,g\}}^n\delta}{2}$-short-hop path between them.
\item Transversal intersection of short-hop paths is preserved under sufficiently small perturbation of the parameter.
\end{enumerate}
\end{lemma}
\begin{proof}
The first part is \cite[Lemma~5.2.2]{calegari}.
The proof of the second part follows \cite[Proposition~7.1.2]{calegari}, only with weaker assumptions and $\abs{z}$ replaced with $\norm{\{f,g\}}$.

For the third part, consider the map that, given an address $w$, changes the parameter, that is the map defined by
$$\lambda \mapsto \pi(w,\lambda).$$ 
This map is continuous. 

Let $e_i,h_i$ be short-hop paths that intersect transversely for a parameter $\lambda$.
These are piecewise linear, with finitely many interim points inside $\A_\lambda$, intersecting transversely at a finite number of points.

For any parameter, we will consider the piecewise-linear path with vertices having the same addresses as the ones in the short-hop paths of the original parameter.
We will show that that under certain requirements, these paths would also be short-hop paths and have the same intersection number.

Indeed, being short hop is described by a number of strict inequalities on the distances of adjacent vertices of the form $$d(\pi_\lambda(e_i), \pi_\lambda(e_{i+1}))<d(f\A_\lambda,g\A_\lambda),$$ which are all satisfied for the original parameter. From the aforementioned continuities, each of the inequalities would be satisfied for any sufficiently close parameter. As there are only finitely many such conditions, we can require all of them. This is equivalent to a single upper bound on the distance of the perturbed parameter.

Similarly, we will impose more requirements. Assume the two original paths intersect transversely between vertices $p_1,p_2$ from the first path, and $q_1,q_2$ from the second path. 
After a small enough perturbation, these linear segments intersect transversely as well. We can impose this fact using a finite number of strict inequality requirements on the pairwise distances of the four points (to maintain the relative positions of the four points). Therefore we can maintain the transversal intersection point.

Each of these requirements becomes an upper bound on the distance of the perturbed parameter. As there are only finitely number of vertices and intersection points, there are a finite number of requirements. 

Therefore, any sufficiently small perturbation, would have these paths being short-hop and have the same intersection number.
\end{proof}
This leads to the following definitions:

\begin{definition}[{\cite[Definition~7.1.3]{calegari}}]\label{def:trap}
	Let $\lambda\in P$ be a parameter (and $\A=\A_\lambda$).
	
	Suppose $D$ is a closed topological disk with $\A\subseteq \interior{D}$.
	
	A pair of finite words $u^n,v^m$ of lengths $n,m$ (we omit the lengths) are called a trap for $(\lambda, D)$ if:
	\begin{enumerate}
		\item $u$ starts with $f$ and $v$ starts with $g$.
		\item There are points $p^{\pm}\in u(\A) \setminus v(D)$ and $q^{\pm}\in v(\A) \setminus u(D)$ and continuous paths $\alpha \subseteq u(D)$ with endpoints $p^\pm$ and $\beta \subseteq v(D)$ with endpoints $q^\pm$ with an algebraic intersection number $\neq 0$.
		\item $d(f(\A),g(\A))<\epsilon$ where $\epsilon$ satisfies $N_{\frac{\epsilon}{2}}(\A)\subseteq D$.
	\end{enumerate}
\end{definition}

\begin{figure}[h]
	\centering
	\begin{subfigure}[h]{0.45\textwidth}
		\centering
		\fbox{\includegraphics[width=\textwidth]{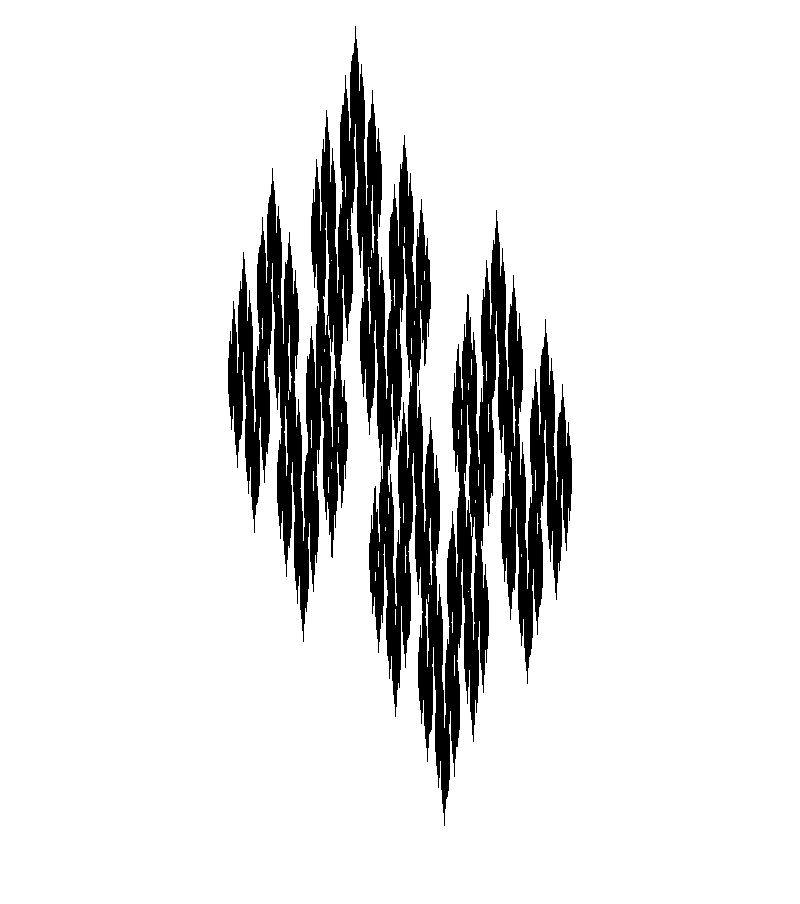}}
		\caption{The attractor from Figure \ref{fig:convex_hull_opp_signs}}
		\label{fig:trap_example_attractor}
	\end{subfigure}
	\hfill
	\begin{subfigure}[h]{0.45\textwidth}
		\centering
		\fbox{\includegraphics[width=\textwidth]{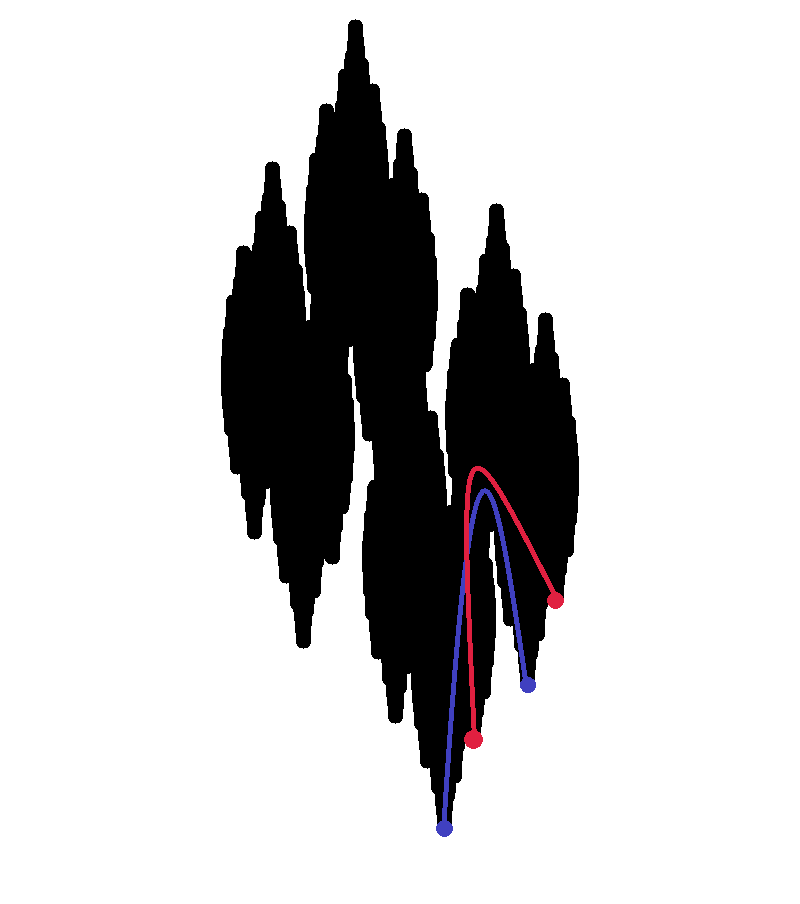}}
		\caption{A small neighborhood of the attractor, with trap words in blue and red}
		\label{fig:trap_example_trap}
	\end{subfigure}
	
	\caption{An example of a trap}
	\label{fig:trap_example}
\end{figure}

The definition of trap depends on a choice of paths $\alpha,\beta$. However, this choice can be avoided:
\begin{lemma}[{\cite[Lemma~7.1.4]{calegari}}]
	Suppose $u,v$ are a trap for $(\lambda,D)$. Let $ p^{\pm}\in u(\A) \setminus v(D)$ and $q^{\pm}\in v(\A) \setminus u(D)$ be as in Definition \ref{def:trap}. Then any paths $\alpha \subseteq u(D)$ with endpoints $p^\pm$ and $\beta \subseteq v(D)$ with endpoints $q^\pm$ must intersect.
\end{lemma}

\begin{theorem}\label{thm:trap_implies_interior}
	Suppose $u^n,v^m$ is a trap for $(\lambda, D)$, then $\lambda\in \interior{\mathcal{L}}$.
\end{theorem}
\begin{proof}
	Denote $L=\norm{\{f,g\}}$ and $\delta = d(f\A, g\A)$. We will prove $\delta = 0$, this will demonstrate $\lambda\in \mathcal{L}$.
	
	By Lemma \ref{thm:short_hop_stability}, there is a short-hop path between every two points in $\A$, that is a piecewise linear path where the distance between each pair of adjacent points is $\leq\delta$ (with interim vertices in $\A$ itself).
	
	$p^\pm$ share a prefix $u^n$ of length $n$. From Lemma \ref{thm:short_hop_properties}, there is a $(L^n\frac{\delta}{2}, u(D))$-short-hop path between $p^\pm$.
	
	Similarly, there is a $(L^m\frac{\delta}{2}, v(D))$-short-hop path between $q^\pm$.
	
	By the intersection condition, these paths intersect at a point $r$ and hence
	
	$$d(u\A, v\A)\leq d(u\A, r) + d(v\A, r) \leq \frac{\delta}{2} L^n + \frac{\delta}{2} L^m \leq \delta L^{\min(n,m)}
	$$
	
	where this bound follows from the fact that the paths are contained in $N_{L^n\frac{\delta}{2}}(u\A)$ and $N_{L^m\frac{\delta}{2}}(v\A)$ respectively.
	Also,
	$$\delta =d(f\A,g\A)\leq d(u\A,v\A)\leq \delta L^{\min(n,m)},$$
	because of the prefixes of $u$ and $v$.
	
	But if $\delta>0$ then $\delta L^{\min(n,m)}<\delta$ in contradiction.
	
	The interior part follows directly from Lemma \ref{thm:short_hop_stability} part 3.
\end{proof}

\subsection{The Homogeneous Self-Affine Case} \label{self_affine}
Still following \cite{calegari}, with a slightly different proof.
The homogeneous self-affine case is where the two maps of the IFS also differ by a constant vector for all parameters. We will assume this fact. As we assume the maps themselves are affine, we write them as $\{T+s_f,T+s_g\}$ for an invertible matrix $T$ and two vectors $s_f,s_g$. In this case we find that $\norm{\{T+s_f,T+s_g\}}=\norm{T}$. 
Notice that cylinder sets of order $m$ are all translates of $T^m\A$. For a finite word $u^m$, we'll denote this translation vector by $s_{u^m}$. This implies cylinder sets of the same order are also translations of each other. The translation vector between $u^m\A$ and $v^m\A$ is $s_{v^m}-s_{u^m}$, and can be expressed as $p(T)(s_f-s_g)$ where $p$ is a polynomial with coefficients in $\{\pm 1, 0\}$. The degree of the polynomial is $m-1-i$ where $i$ is the length of the common prefix of the two words.

It turns out that in the homogeneous case, there is a characterization of traps in terms of simpler conditions.
First, for convenience, we normalized these translation vectors:
\begin{definition}
	A normalized translation vector is a vector of the form $w=T^{-m}(s_{u^m}-s_{v^m})$ where $u^m$ and $v^m$ are words of length $m$, which do not share a common prefix. Equivalently, $u^m\A = v^m\A + T^mw$.
\end{definition}

The trap condition is forcing an intersection, in the case where these two cylinder sets are translates of each other, it suffices to show they alternate along the boundary. This leads to the following definition (still following \cite{calegari}):

\begin{definition}[{\cite[Definition~7.2.1]{calegari}}]\label{def:trap_like}
	A compact set $X\subset \mathbb{R}^2$ is cell-like if its complement is connected. A vector $w\in\mathbb{R}^2$ is called trap-like for $X$ if 
	\begin{enumerate}
		\item $X\cap (X+w)\neq\emptyset$ i.e $X\cup (X+w)$ is connected.
		\item There are four points in the boundary of $X\cup (X+w)$ which alternate between $X\setminus (X+w)$ and $(X+w)\setminus X$.
	\end{enumerate}
\end{definition}

\begin{lemma}[{\cite[Lemma~7.2.2]{calegari}}]\label{thm:trap_like_non_convex}
	Suppose $X$ is cell-like. Then there exists a trap-like vector iff $X$ is not convex.
\end{lemma}

\begin{lemma} \label{thm:trap_like_verbose}
	Suppose $X$ is compact and non convex. Let $l$ be a line segment on the convex hull of $X$ with only a pair of intersection points with $X$ say $p,q$. For every $r\in l^o$ there exists an open ball $B(r,\epsilon_r)\subseteq X^c$ such that for every $v\in B(r,\epsilon_r)$ which is on the same side of $l$ in which $X$ resides, the vector $w:=v-p$ is trap-like for $X$.
\end{lemma}
\begin{proof}
	The existence of the ball is guaranteed because both $l$ and $X$ are compact, and for the trap-like property the proof is exactly the same as the proof in \cite[Lemma~7.2.2]{calegari}.
\end{proof}
\begin{remark}
	We write this verbose version of the existence of trap-like vectors because it can be more useful when we know the convex hull of the attractor explicitly. A potential usage of this lemma is to find cusp corners in $\N$. See Section \ref{sec:questions} for more details.
\end{remark}

To go from trap-like vectors to traps, we need to start with a connected set. A potential issue is that, a priori, the starting parameter may not be in the locus. To address this, we can start with a parameter in the locus, i.e.\  with a connected attractor, take an $\epsilon$-neighborhood of the attractor and then perturb the parameter. This method is used in the proof of the existence of traps near a given parameter in the locus.

It turns out trap-like vectors for an $\epsilon$-neighborhood, which are also translation vectors between cylinder sets of the same order are traps:
\begin{theorem}\label{thm:trap_in_affine_case}
	Suppose $\A$ is an attractor and $X$ is filling any holes in $\A$.
	Let $\epsilon>0$ such that $D=N_\epsilon(X)$ is connected (and in particular cell-like). Let $w$ be both trap-like and a normalized translation vector between the cylinder sets $u^m\A, v^m\A$ of order $m$ which share no common prefix. Then there exists a trap for this parameter using $D$ as the disk.

\end{theorem}
\begin{proof}
	Let $p_1,p_2,q_1,q_2$ be the four points implied in Definition \ref{thm:trap_like_non_convex}. These lie on the boundary of $D$, which is contained in $\A$, therefore we can apply the maps $u^m,v^m$ on them and remain in $\A$.
	The trap consists of the points $v^m(p_i), v^m(q_i)$ and the words $u^m,v^m$.
	
	First, $\A\subseteq \interior{D}$ is trivially true.
	We now prove $v^m(p_i)\in v^m(\A)\setminus u^m(D)$:
	First, $p_i\in \A \setminus N_\epsilon(X+w),$
	therefore $v^m(p_i)\in v^m\A$. Assume, for the sake of contradiction, that $v^m(p_i)= u^m(x)$ for some point $x$ that satisfies $d(x,X)\leq\epsilon$.
	
	Expanding the definitions of $v^m,u^m$ as maps we obtain
	$$
	T^m\cdot p_i + s_{v^m}= T^m \cdot x + s_{u^m}.
	$$
	
	Since $w$ is a normalized translation vector, we can rearrange and simplify this equation and arrive at
	
	$$p_i = x +w.$$
	
	This is in direct contradiction with the assumption $p_i\notin N_\epsilon(X+w)$.
	
	Similarly, $v^m(q_i)\in u^m(\A) \setminus v^m(D)$.
	Therefore $v^m(p_i), v^m(q_i)$ are in the correct sets.
	
	Since $N_\epsilon(X)$ is connected, we find that $\epsilon\geq d(f\A,g\A)$. Hence, by Lemma \ref{thm:short_hop_properties}, there exists short-hop paths between $p_1,p_2$ and $q_1,q_2$. Because the points are alternating, these paths must intersect. Without loss of generality, we can assume the intersection is transversal (otherwise, apply a small perturbation), and therefore by the the third part of Lemma \ref{thm:short_hop_properties} we obtain intersecting paths for $v^m(p_i), v^m(q_i)$ as well. Putting this together, we constructed a trap.
\end{proof}

The condition of trap-like is verified easily, we only need to find one which is a normalized translation vector. 

For the proof of regular closedness, \cite{calegari} introduced the Surjective Perturbation Lemma  which shows every vector is a translation vector for some nearby parameter. Combining this with existence of trap-like vectors yield traps nearby to the parameter we started with. And as traps are in the interior, this implies regular closedness.

Unfortunately, the Surjective Perturbation Lemma is not absolutely generalizable, and for the case of $\N$ it is not generally true. We will work around this in the next section.
\section{Applying the Method of Traps for $\N$}\label{sec:N}
Finding trap-like vectors is straightforward for $\N$, because the attractor is non-convex in all but trivial cases (Corollary \ref{thm:non_convexity_of_attractor}). Therefore the existence of trap-like vectors is guaranteed. As a result, we only need to characterize when trap-like vectors are also normalized translation vectors. We accomplish this using an analog of the Surjective Perturbation Lemma, similar to \cite{calegari}. 

\cite[Corollary~7.2.6]{calegari} proves every vector is a normalized translation vector for some nearby parameter. By combining this with the existence of trap-like vectors, Calegari et al.\ deduced the existence of parameters with traps nearby to any parameter in $\M\setminus \M_t$. This is their proof of regular closedness.

We aim to follow a similar path. However, the analogous statement of the Surjective Perturbation Lemma fails in general in $\N$. The core reason for the failure is that for real analytic functions, a slight perturbation may remove zeros -- unlike in the case of complex analytic functions. We address this issue by considering only the subset of the parameters for which the number of zeros is preserved. For this subset, the proof of regular closedness is very similar to the one in \cite{calegari}. The goal of this section is to prove this analog regular closedness.

Furthermore, this phenomenon of disappearing zeros only affects zeros of higher order, most of which, fortunately, lie within the trivial region. This is manifested in the following lemma:
\begin{lemma}\label{thm:order_of_zeros}
	Let $(\gamma,\lambda)\in \N \setminus(\N_t \cup \text{Diag}((-1,1)))$, and let $f\in\B$ with $f(\gamma)=f(\lambda)=0$. Then
\begin{enumerate}
	\item  $f$ has at most one zero of multiplicity $\geq 2$ in $(-1,1)$.
	\item If $0<\gamma<\lambda$, then $\gamma$ is of multiplicity 1 for $f$.
\end{enumerate}
\end{lemma} 

\begin{proof}
Denote $\alpha_i$ the smallest real number such that there exists an $f\in\B$ with zero of order $i$ at $\alpha_i$.
Following \cite[Theorem~2.6]{shmerkin_sol2000}, $\alpha_2> 0.668$. 

Assume $\gamma$ is a zero of order $m$ and $\lambda$ a zero of order $k$ for $f$, and without loss of generality that $\abs{\lambda}>\abs{\gamma}$. 
For the first part, we assume $m,k\geq 2$ and we'll show $\abs{\gamma\lambda}\geq\frac{1}{2}$.

First note that from \cite[Theorem~2.6]{shmerkin_sol2000}, we must have $\abs{\gamma}\geq \alpha_2 > 0.668$. Indeed, if $\gamma>0$ it follows directly, for $\gamma<0$ replace $f(x)=1+\sum_{n=1}^\infty a_nx^n$ with $h(x)=1+\sum_{n=1}^\infty a_{2n}x^{2n} - \sum_{n=1}^\infty a_{2n-1}x^{2n-1}$ and get $h(-\gamma)=h'(-\gamma)=0$, which implies $\abs{\gamma}=-\gamma\geq\alpha_2$ as well. 

We recall from \cite[Theorem~2]{beaucoup} or \cite[Theorem~2.4]{shmerkin2006overlapping} that the product of the magnitudes of all $n$ (possibly complex) zeros $\beta_i$ of $f$ in a neighborhood of zero, counted with multiplicity, must satisfy 
$$
\prod_{i=0}^{n}\abs{\beta_i}\geq\left(1+\frac{1}{n}\right)^{-\frac{n}{2}}(1+n)^{-\frac{1}{2}}.
$$ Denote the right hand side by $C(n)$.
Let $n$ be the total number of zeros of $f$ in $\overline{B(0,\abs{\lambda})}$ (a closed disk in the complex plane), counting with multiplicity (including conjugate complex zeros as well). Then we must have
$$
\abs{\lambda}^n\geq \prod_{i=1}^{n}\abs{\beta_i} \geq C(n).
$$
Therefore
$$
\abs{\lambda}\geq C(n)^{\frac{1}{n}}.
$$
Assume, for the sake of contradiction that $\abs{\gamma\lambda}<\frac{1}{2}$. 
Denote
$$
D(x)=C(x)^{\frac{1}{x}}=\left(\left(1+\frac{1}{x}\right)\cdot (1+x)^{\frac{1}{x}}\right)^{-\frac{1}{2}}.
$$
For positive $x$, the expression $1+\frac{1}{x}$ is decreasing, and so is $(1+x)^{\frac{1}{x}}$ (this can be checked using the derivative). Thus, $D$ must be increasing. 

If $n\geq 5$, we get
$$\abs{\lambda}\geq D(n)\geq D(5)=C(5)^{\frac{1}{5}}.$$
This implies
$\abs{\gamma}<\frac{1}{2\abs{\lambda}}\leq\frac{1}{2C(5)^{\frac{1}{5}}}\approx 0.655$, contradicting $\abs{\gamma}\geq\alpha_2$. Hence the non-triviality assumption was wrong, i.e.\ $\abs{\gamma\lambda}\geq\frac{1}{2}$.

If $n=4$, we must have $m=k=2$, and there are no other zeros of $f$ in $\overline{B(0,\abs{\lambda})}$. Therefore
$$
\abs{\gamma\lambda}^2= \prod_{i=1}^{n}\abs{\beta_i} \geq C(4)\approx 0.286 >\frac{1}{4}.
$$
Therefore we must have
$$
\abs{\gamma\lambda}>\frac{1}{2}
$$
contradicting non-triviality again.

For the second part, either both zeros are simple, or exactly one of them is simple. Assume, for the sake of contradiction, that $\lambda$ is the simple zero (that is, $m=1$) and $\gamma$ is of order $k$ for some $k\geq 2$. In particular, $\gamma\geq 0.668$.
We assume $(\gamma,\lambda)$ is not trivial, i.e.\ $\gamma\lambda<\frac{1}{2}$.

Hence, we find $\lambda\leq 0.5/0.668\approx 0.7485$.
We follow the notation in \cite[Proposition~2.5]{N_general} along with the associated table: the function $\psi$ describes a lower bound to the third smallest zero (counting with multiplicity) of any $f\in\B$.
We know $\gamma<\frac{1}{\sqrt{2}}\approx 0.707<\alpha_3\approx 0.7278$, therefore $\psi(\gamma)$ is well defined.

$\psi$ is decreasing, and we know $\psi(\gamma)\leq \lambda$. Therefore $\psi(\gamma)\leq 0.7485$ which implies $\gamma\geq\psi^{-1}(0.7485)\geq 0.67$ (based on the table). 
Repeating the argument, we find $\lambda\leq0.5/0.67\approx 0.7463$ and therefore $\gamma\geq \psi^{-1}(0.7463)\geq 0.69$ (based on the table).
Repeating again, $\lambda\leq0.5/0.69\approx 0.7246$. Since $0.7246<\alpha_3\approx 0.7278$, this is a contradiction, as $\psi(\gamma)\geq \alpha_3$ by the definition of $\psi$.

Thus, any such point is in the trivial region of $\N$.
\end{proof}

Therefore, for any non-trivial purpose, we need to consider only points where one of the zeros is simple.

The proof of regular closedness is comprised of three main parts: demonstrating the non-convexity of the attractor, applying the Surjective Perturbation Lemma to a subset of $\N$ (and characterizing this subset), and describing normalized translation vectors for nearby parameters.

\subsection{Convex Hull of the Attractor for $\N$}
\begin{theorem}\label{thm:convex_hull_opp}
	Assume $\gamma<0<\lambda$ and $\abs{\gamma}<\abs{\lambda}$. Then the convex hull of $\A$ is described as follows:
\begin{figure}[h]
	\centering
	\fbox{\includegraphics[scale=0.7]{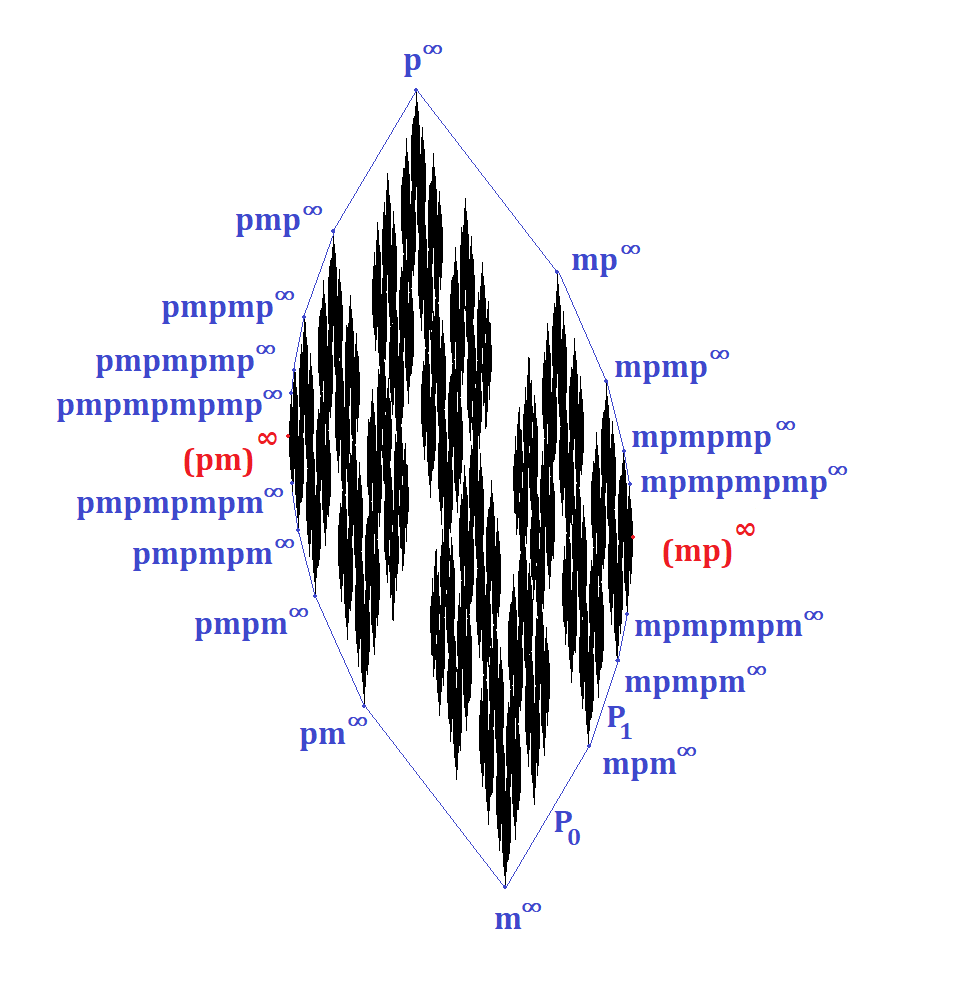}}
	\caption{$\A_{(-\frac{10}{17},\frac{10}{13})}$ together with the vertices and edges of the convex hull}
	\label{fig:convex_hull_opp_signs}
\end{figure}
\begin{enumerate}
	\item The rightmost vertex is $\pi((mp)^\infty)$. 
	\item The leftmost vertex is $\pi((pm)^\infty)$.
	\item All the other vertices are given by $$A_k=\pi((mp)^km^\infty), B_k=\pi((pm)^km^\infty), C_k=\pi((pm)^kp^\infty), D_k=\pi((mp)^kp^\infty),$$ where $k=0,1,2,\dots$. Their relative order is represented in Figure \ref{fig:convex_hull_opp_signs}.
\end{enumerate}

\end{theorem}
\begin{proof}
We fix some notation:
For any word $w$ formed from the alphabet $\{m,p\}$, write $g_w(x)=\sum_{i=1}^\infty w_ix^i$ and $\pi(w)=(g_w(\gamma),g_w(\lambda))$ where $w_i=1$ is represented by $p$ ("plus"), and $w_i=-1$ by $m$ ("minus").

The top corner is attained by maximizing $g_w(\lambda)$. Since $\lambda$ is positive, the maximal value is for all coefficients being $+1$.

The right corner is attained by maximizing $g_w(\gamma)$. Since $\gamma$ is negative,
the maximal value is for all even coefficients being $+1$ and all odd coefficients being $-1$.

Due to the symmetry of the attractor around zero, that is, $(x,y)\in \A$ iff $(-x,-y)\in\A$, the characterization of the top vertices would follow from the bottom ones.

The method of proof here is similar to \cite[Theorem~2.1]{hare_sidorov_2017}:
We denote the line segment connecting $A_k,A_{k+1}$ by $P_k$.
First, we prove all points of the attractor are strictly above the line segment connecting $A_0=\pi(m^\infty)$ and $A_1=\pi(mpm^\infty)$ except these two points.

The point with minimal $y$-coordinate in the attractor is ${A_0=\pi(m^\infty)}$. Therefore any point $\pi(w)$ is to the right of $A$ iff the slope of the line segment connecting it with $A$ is positive. We denote this slope by $s_w$.

A useful fact we will use repeatedly: if we write the coordinates of ${\pi(w)-\pi(m^\infty)}$ as $(a,b)$ then $b\geq a$ with equality only for $w=m^\infty$. Indeed, by definition $$a=\sum_{i=1}^{\infty}(w_i+1)\gamma^i$$ and $$b=\sum_{i=1}^{\infty}(w_i+1)\lambda^i.$$
These are geometric sums with non-negative coefficients (as $w_i+1\in\{0,2\}$) and $\abs{\gamma}<\lambda$. The only case of equality would be if all coefficients were $0$, i.e.\ if $w=m^\infty$.

We will prove the inequality $s_w>\frac{\lambda^2}{\gamma^2}$ for all words $w$ which are to the right of $A_0$. Indeed, write $w=w_1w_2w'$ and the slope is
$$
s_{w_1w_2w'}=\frac{(w_1+1)\lambda+(w_2+1)\lambda^2+\lambda^2b}{(w_1+1)\gamma+(w_2+1)\gamma^2+\gamma^2a}
$$
where $(a,b)$ are the coordinates of $\pi(w')-\pi(m^\infty)$. The inequality $s_w>s_{mpm^\infty}$ becomes
$$
\frac{(w_1+1)\lambda+(w_2+1)\lambda^2+\lambda^2b}{(w_1+1)\gamma+(w_2+1)\gamma^2+\gamma^2a}>\frac{\lambda^2}{\gamma^2}
$$
as we assume the slopes are all positive, we can multiply by the denominators without changing the direction of the inequality. We cancel the term ${(w_2+1)\lambda^2\gamma^2}$ which appears on both sides, and the inequality becomes
$$
(w_1+1)\gamma^2\lambda+\gamma^2\lambda^2 b >(w_1+1)\gamma\lambda^2+\gamma^2\lambda^2 a.
$$
But $b\geq a$ with equality only for $w'=m^\infty$, therefore we obtain the inequality $\gamma^2\lambda^2 a\geq \gamma^2\lambda^2 b$, with equality only for $w'=m^\infty$.
Also, $w_1+1\in\{0,2\}$ and $\lambda>0>\gamma$ therefore $(w_1+1)\gamma^2\lambda\geq (w_1+1)\gamma\lambda^2$ with equality only for $w_1=-1$.

Therefore this inequality is satisfied, unless $w'=m^\infty$ and $w_1=-1$. We assumed $w\neq m^\infty$, therefore the only possible exception must also satisfy ${w_2=1}$, i.e. $w=mpm^\infty$ as required.

Second, we claim all points in the attractor are above the line segment $P_k$ connecting $A_k,A_{k+1}$ for all $k$.
The slope of this line is 
$
\frac{\lambda^{2k+2}}{\gamma^{2k+2}}
$
which is increasing in $k$. We will prove this by induction on $k$:

For $k=0$, we already proved it. 

Consider a word not equal to either $(mp)^km^\infty$ or $(mp)^{k+1}m^\infty$. We may assume without loss of generality that $\pi(w)$ is to the right of $P_k$. Otherwise, there are two cases: First, if the point is to the left of $A_0$, in which case it is in a quadrant (relative to $A_0$) directly above $P_k$ for all $k$. Second, the point is to the right of $A_0$, in which case there exists $k'<k$ such that $\pi(w)$ is to the right of $A_{k'}$ and to the left of $A_{k'+1}$. By induction, $\pi(w)$ will be above the line $P_{k'}$. But the slopes are increasing, therefore $\pi(w)$ is also above $P_k$.

Therefore we may assume $\pi(w)$ is to the right of $A_k$. Write $w=w_1w_2\dots w_{2k+2}w'$. We will prove the slope of the line segment connecting $A_k,\pi(w)$ is larger than that of $P_k$, which imply the result. This is the inequality
$$
\frac{\sum_{i=1}^{k+1}(w_{2i-1}+1)\lambda^{2i-1}+\sum_{i=1}^{k}(w_{2i}-1)\lambda^{2i}+(w_{2k+2}-1)\lambda^{2k+2}+\lambda^{2k+2}b}{\sum_{i=1}^{k+1}(w_{2i-1}+1)\gamma^{2i-1}+\sum_{i=1}^{k}(w_{2i}-1)\gamma^{2i}+(w_{2k+2}-1)\gamma^{2k+2}+\gamma^{2k+2}a}>\frac{\lambda^{2k+2}}{\gamma^{2k+2}}
$$
where $(a,b)$ are the coordinates of the vector $\pi(w')-\pi(m^\infty)$. Due to the assumptions, both denominators are positive. We can multiply by them without changing the direction of the inequality. We cancel the term $(w_{2k+2}-1)\gamma^{2k+2}\lambda^{2k+2}$ which appears on both sides, and the inequality becomes:
\begin{eqnarray*}
&\sum_{i=1}^{k+1}(w_{2i-1}+1)\lambda^{2i-1}\gamma^{2k+2}+\sum_{i=1}^{k}(w_{2i}-1)\lambda^{2i}\gamma^{2k+2}+\gamma^{2k+2}\lambda^{2k+2}b \ &> \\
&\sum_{i=1}^{k+1}(w_{2i-1}+1)\gamma^{2i-1}\lambda^{2k+2}+\sum_{i=1}^{k}(w_{2i}-1)\gamma^{2i}\lambda^{2k+2}+\gamma^{2k+2}\lambda^{2k+2}a &
\end{eqnarray*}

Similar to the base case, we find that $$\gamma^{2k+2}\lambda^{2k+2}b\geq \gamma^{2k+2}\lambda^{2k+2}a$$ with equality only for $w'=m^\infty$.

Because $\abs{\gamma}<\abs{\lambda}$, for all $1\leq i\leq k$ we find that
$$
(w_{2i}-1)\lambda^{2i}\gamma^{2k+2}\geq (w_{2i}-1)\gamma^{2i}\lambda^{2k+2}
$$
with equality only for $w_{2i}=1$.

Similarly, for all $1\leq i\leq k+1$ we find that
$$
(w_{2i-1}+1)\lambda^{2i-1}\gamma^{2k+2} \geq
(w_{2i-1}+1)\gamma^{2i-1}\lambda^{2k+2}
$$
with equality only for $w_{2i-1}=-1$.

Therefore the desired inequality is satisfied, unless all the above cases are equalities. But this implies $w=(mp)^km^\infty$ or $w=(mp)^{k+1}m^\infty$ (as all coefficients of $w$ are determined except for $w_{2k+2}$ which canceled out).

This finishes the proof of the vertices to the right of $A_0$. The vertices to the left of $B_0$ are characterized in exactly the same manner, with all inequalities reversed as the slopes are 
$\frac{\lambda^{2k+1}}{\gamma^{2k+1}}$
which are negative and decreasing.
\end{proof}
\begin{corollary}\label{thm:non_convexity_of_attractor}
	If $(\gamma,\lambda)\in \N\setminus(\N_t\cup \text{Diag}((-1,1)))$ then the attractor $\A_{(\gamma,\lambda)}$ is not convex.
\end{corollary}
\begin{proof}
If $\gamma,\lambda$ have the same sign, then this result follows from \cite[Theorem~2.1]{hare_sidorov_2017}, which characterizes the convex hull and its vertices. In particular, all points of the attractor are above the line segment connecting $m^\infty$ and $pm^\infty$ (except these two themselves).

Assume without loss of generality, that $\gamma<0<\lambda$. If not, replace $(\gamma,\lambda)$ with $(\lambda,\gamma)$. We can similarly assume $\abs{\gamma}<\abs{\lambda}$; indeed, first note that in the non-trivial region of $\N$ there are no points with $\gamma=-\lambda$. Next, if $\abs{\gamma}>\abs{\lambda}$, replace $(\gamma,\lambda)$ with $(-\lambda,-\gamma)$.

Therefore the convex hull from Theorem \ref{thm:convex_hull_opp} applies. In particular, there are no points in the attractor on the line segment joining $\pi(m^\infty)$ and $\pi(mpm^\infty)$ except for these two points themselves. Thus, the attractor is not convex.
\end{proof}

\subsection{$\R^2$ Version of the Surjective Perturbation Lemma}

The Surjective Perturbation Lemma in \cite{calegari} is a key result enabling finding parameters for which a trap exists. This allowed the authors of \cite{calegari} to prove regular closedness for $\M$.
Unfortunately, this lemma does not hold in $\N$ in general, which reduces the scope of our main result. The version that does hold is partial and will be described below.

\begin{definition}\label{def:N_prime}
	Let $\N'$ denote the set of all parameters $(\gamma,\lambda)\in \mathbb{R}^2$ for which there exists a power series $f\in\B$ such that $f(\lambda)=f(\gamma)=0$, and both $\gamma$ and $\lambda$ are sign changes for $f$.
\end{definition}
\begin{definition}\label{def:property_U}
	Let $f:[a,b]\to\R$ be a real power series that has an isolated zero $\lambda\in(a,b)$. Let $f_n$ be the truncation of $f$ to an $n$-degree polynomial. We say $f$ satisfies property U (unique) around $\lambda$, if for any sufficiently small $\epsilon>0$, there exists an $N$ such that for all $n>N$, $f_n$ has a unique zero $\lambda_n$ in $B(\lambda,\epsilon)$, and this neighborhood contains no critical points of $f_n$ except possibly $\lambda_n$ itself.
\end{definition}
\begin{definition}\label{def:N_u}
	Let $\N'_U$ be the set of all parameters $(\gamma,\lambda)\in\N'$ such that there exists $f\in\B$ that satisfies $f(\gamma)=f(\lambda)=0$ and $f$ satisfies property U around both $\gamma$ and $\lambda$.
\end{definition}

\begin{lemma}[Surjective Perturbation for $\N'_U$]\label{thm:surjective_perturbation_N}
	Let $(\gamma_0,\lambda_0)\in {\mdiags{\N'_U}}$ and let $f\in\B$ as in Definition \ref{def:N_u}. Suppose $u,v$ are distinct infinite words such that $f$ is induced from $u-v$. In particular $\pi(u,(\gamma_0,\lambda_0))=\pi(v, (\gamma_0,\lambda_0))$. 
	Let $f_n$ be the truncations of $f$, then by abuse of notation, it is also the translation vector of cylinder sets of $u,v$:
	$$
	u^n\A_{(\gamma_0,\lambda_0)}=v^n\A_{(\gamma_0,\lambda_0)} +f_n(\gamma_0,\lambda_0)
	$$
	(where $u^n,v^n$ are the $n$-prefixes of the words). With the notation in Section \ref{self_affine} we know $f_n(\gamma,\lambda)=(s_{v^n}-s_{u^n})(\gamma,\lambda)$.
	Then for any sufficiently small $\epsilon>0$ there is $\delta>0$ (independent of $n$) and an integer $N$ such that for any $n\geq N$: $$f_n(B((\gamma_0,\lambda_0),\epsilon))\supseteq B(0,\delta).$$
	
\end{lemma}
For the proof we need some auxiliary results:
\subsubsection{Auxiliary results required}
These are basic facts from analysis, which we need for the proof. The proofs are technical and therefore are left to Section \ref{sec:aux_proofs}.

\begin{lemma}\label{thm:B_completeness}
	If $f_n\in\B$ has a uniform limit $f$, then $f\in \B$.
\end{lemma}
\begin{lemma}\label{thm:simple_is_U}
	Let $f$ be real analytic. If $\lambda$ is a simple zero of $f$, then it satisfies property U.
\end{lemma}

\begin{lemma}\label{thm:unqiue_implicit_is_continuous}
	Suppose $K\subseteq \R^k$ is closed and $M\subset \R^m$ be compact. Let $F:K\times M \to \R^n$ be continuous. Assume that for every $x\in K$ there exists a unique $y\in M$ such that $F(x,y)=0$. Denote the mapping $x\mapsto y$ by a function $g:K\to M$. Then $g$ is continuous.
\end{lemma}
\begin{lemma}\label{thm:not_U_lower_order}
	Let $f$ be non-constant real analytic with an isolated real zero $\lambda$ which is a sign change, of order $k$. Assume $f$ does not satisfy property U, then for a subsequence $n_m$, there are zeros $\lambda_{n_m}$ of $f_{n_m}$ such that $\lambda_{n_m}\to \lambda$, of odd order $k_m$ that is strictly smaller than $k$.
\end{lemma}

\begin{lemma}\label{thm:B_compactness}
	In any interval $[a,b]\subset (-1,1)$ where $a<0<b$, any sequence $f_n\in\B$ contains a uniformly convergent subsequence $f_{n_k}\rightrightarrows f\in \B$.
\end{lemma}

\subsubsection{Proof for subset of $\N$}

\begin{proof}[Proof of Lemma \ref{thm:surjective_perturbation_N}]
	$f_n$ is the $n$-th truncation of the power series corresponding to the representations of $u,v$. In particular, $f_n$ uniformly converges to $f$, which is not identically zero as $u\neq v$.
	Without loss of generality, assume that $\gamma_0<\lambda_0$ and that $f$ changes sign from $+$ to $-$ at $\gamma_0$.
	
	Due to property U, there exists a neighborhood of $\gamma_0$ for which $f_n$ is monotonically decreasing for all large enough $n$, denote it by $[\alpha,\beta]$. Clearly, $f$ is monotonically decreasing in this neighborhood as well.
	
	Let $\epsilon>0$ be such that $\alpha<\gamma_0-\epsilon$. Then take $\delta_1=f(\gamma_0-\epsilon)$. Let $0<w_1<\delta_1$ be arbitrary. As $f_n(\gamma_0-\epsilon)\to \delta_1$, for sufficiently large $n$ we must have $f_n(\gamma_0-\epsilon)>w_1$. Similarly, $f_n(\gamma_0)\to f(\gamma_0)=0$, therefore for sufficiently large $n$ we also have $f_n(\gamma_0)<w_1$. Hence, by the Intermediate Value Theorem (and monotonicity of $f_n$), there is a point $\gamma\in(\gamma_0-\epsilon,\gamma_0)$ with $f_n(\gamma)=w_1$.
	
	Therefore every $0<w_1<\delta_1$ is achieved for every sufficiently large $n$ in $(\gamma_0-\epsilon,\gamma_0)$.
	
	An analogous argument to the right of $\gamma_0$ yields that every $-\delta_2<w_1<0$ is also achieved in $(\gamma_0,\gamma_0+\epsilon)$ where we require $\gamma_0+\epsilon<\beta$ and set $\delta_2=-f(\gamma_0+\epsilon)$.
	
	Regarding $w_1=0$, it is achieved due to Lemma \ref{thm:case_12_tail} part 1, if we choose $n$ sufficiently large such that $\gamma_n$ falls in $(\gamma_0-\epsilon,\gamma_0+\epsilon)$.
	
	Repeating this argument for $\lambda_0$ yields an analogous $\delta_3$ and $\delta_4$. Therefore by taking a sufficiently small $\delta$ would satisfy all six cases. Hence, every vector $w\in\R^2$ with $\norm{w}_\infty <\delta$ is achieved around $(\gamma_0,\lambda_0)$ for all sufficiently large $n$ as required.
\end{proof}

\subsection{Regular Closedness of a subset}
The final ingredient for proving regular closedness is an analog to \cite[Corollary~7.2.6]{calegari}. The case of $\N$ is more complicated, and requires additional assumptions:

\begin{lemma}[The existence of normalized translation vectors] \label{thm:surj_for_trap_like}
	Let ${(\gamma_0,\lambda_0)\in \N'_U}$. Then for every vector $0\neq w\in\R^2$ and every sufficiently small $\epsilon>0$ there exists a parameter $(\gamma_1,\lambda_1)\in B((\gamma_0,\lambda_0),\epsilon)$ such that $w$ is a normalized translation vector for $(\gamma_1,\lambda_1)$.
\end{lemma}
\begin{proof}
	Let $\epsilon>0$ be sufficiently small, such that $\gamma,\lambda$ are the only zeros of $f$ in $B((\gamma,\lambda),\epsilon)$.
	Suppose $\delta, M_0$ are given by Lemma \ref{thm:surjective_perturbation_N} for $\frac{\epsilon}{2}$.
	Let $V=\overline{B((\gamma_0,\lambda_0),\frac{\epsilon}{2})}$. 
	There is an $M_1$ such that $\norm{T_{(\gamma,\lambda)}^{m}w}<\delta$ for all $m>M_1$ and all $(\gamma,\lambda)\in V$ (this is true as $\norm{T_{(\gamma_0,\lambda_0)}}<1$ and $\epsilon$ is sufficiently small).
	
	Due to property U, there exists an $M_2$ such that for all $m>M_2$, there are unique zeros $\gamma_m$ and $\lambda_m$ of $f_m$ inside $V$. In particular, $f_m$ is injective on $B(\gamma_m,\frac{\epsilon}{2})$ and on $B(\lambda_m,\frac{\epsilon}{2})$.
	
	Let $(\gamma,\lambda)\in V$ and let $M=\max\{M_0,M_1, M_2\}$. Note that $T_{(\gamma,\lambda)}^Mw\in B(0,\delta)$, therefore by Lemma \ref{thm:surjective_perturbation_N} we can choose $(\gamma',\lambda')\in V$ such that $${T_{(\gamma,\lambda)}^Mw=f_M((\gamma',\lambda'))}$$
	We construct a continuous function $g:V\to V$ that maps $(\gamma,\lambda)$ to $(\gamma',\lambda')$.
	We claim $(\gamma',\lambda')$ is unique in $V$. Indeed, $f_M$ is injective on the projections of $V$ to neighborhoods of $\gamma_0$ and $\lambda_0$. Therefore if we define $F(x,y)=f_M(y)-x^Mw$, this function satisfies Lemma \ref{thm:unqiue_implicit_is_continuous} for $K=M=V$ and $n=2$. Therefore, we define a continuous $g:V\to V$ that satisfies 
	$$
	T^M_{(\gamma,\lambda)}w=f_M(g(\gamma,\lambda)).
	$$
	Since $V$ is compact and convex, we can apply Brouwer fixed point theorem and find a parameter $(\gamma_1,\lambda_1)\in V$ such that $g(\gamma_1,\lambda_1)=(\gamma_1,\lambda_1)$. This implies that
	$$
	T^M_{(\gamma_1,\lambda_1)}w=f_M(g(\gamma_1,\lambda_1))=f_M(\gamma_1,\lambda_1).
	$$ 
	Equivalently, $w$ is a normalized translation vector for $(\gamma_1,\lambda_1)$.
\end{proof}

In conclusion, we obtain a partial result concerning regular closedness that is analogous to \cite[Theorem~7.2.7]{calegari}:

\densetraps

\begin{proof}
	The proof is very similar to \cite{calegari} Theorem 7.2.7, where we replaced the three ingredients described in the start of the section by analogs in $\N$.
	
	Suppose $X_{(\gamma_0,\lambda_0)}$ is the region bounded by $\A_{(\gamma_0,\lambda_0)}$, that is, the complement of the unbounded component of the complement of $\A$, in particular it is cell-like. It is also non-convex by Corollary \ref{thm:non_convexity_of_attractor}. Then, by Lemma \ref{thm:trap_like_non_convex} there exists a trap-like $w$ and four alternating points $p_1,p_2\in X_{(\gamma_0,\lambda_0)}$ and $q_1,q_2\in X_{(\gamma_0,\lambda_0)}+w$ on the boundary of $X_{(\gamma_0,\lambda_0)}\cup(X_{(\gamma_0,\lambda_0)}+w)$. Since $\partial X_{(\gamma_0,\lambda_0)}\subseteq\A_{(\gamma_0,\lambda_0)}$, we find that the points are in the attractor, i.e., $p_i\in \A_{(\gamma_0,\lambda_0)}$ and $q_i\in \A_{(\gamma_0,\lambda_0)}+w$.

	We'll denote by $p_i(\gamma,\lambda),q_i(\gamma,\lambda)$ the points with the same addresses as $p_i,q_i$, only for the parameter $(\gamma,\lambda)$.
	For every sufficiently small $\epsilon>0$ we know $N_\epsilon(X_{(\gamma_0,\lambda_0)})$ is connected, $p_1,p_2\in \A_{(\gamma_0,\lambda_0)} \setminus N_\epsilon(X_{(\gamma_0,\lambda_0)}+w)$, and 
	$q_1,q_2\in (\A_{(\gamma_0,\lambda_0)}+w)\setminus N_\epsilon(X_{(\gamma_0,\lambda_0)})$.

	This is because all these conditions are equivalent to inequalities on distances between points or cylinder sets, which are satisfied for $\epsilon=0$ (and there are only a finite number of conditions).
	Fix $\epsilon$. For every sufficiently close parameter $(\gamma,\lambda)$ to $(\gamma_0,\lambda_0)$ the set ${N_\epsilon(X_{(\gamma,\lambda)})}$ is connected, $p_1(\gamma,\lambda),p_2(\gamma,\lambda)\in {\A_{(\gamma,\lambda)}\setminus N_\epsilon(X_{(\gamma,\lambda)}+w)}$ and $q_1(\gamma,\lambda),q_2(\gamma,\lambda)\in {(\A_{(\gamma,\lambda)}+w)\setminus N_\epsilon(X_{(\gamma,\lambda)})}$. Furthermore, we demand that the relative position of the points remains unchanged.
	All these claims can be expressed through (non-strict) inequalities concerning the distances of points or cylinder sets, which all hold for the parameter $(\gamma_0,\lambda_0)$. Due to continuity regarding the parameter of all the maps involved, the claims would be satisfied for any sufficiently close parameter as well.
	
	Since $(\gamma_0,\lambda_0)$ is in $\N'_U$, we can apply Lemma \ref{thm:surj_for_trap_like} and obtain an $m$ and a parameter ${(\gamma_1,\lambda_1)\in B((\gamma_0,\lambda_0),\delta)}$, where $w$ is a normalized translation vector for $u^m,v^m$, i.e.\ 
	$$
	T^{-m}(s_{u^m}-s_{v^m})((\gamma_1,\lambda_1))=w.
	$$
	All the conditions to Lemma \ref{thm:trap_in_affine_case} are satisfied for $(\gamma_1,\lambda_1)$ and this implies a trap. 
	Consequently, we find interior points of $\N$ that are within a distance $\delta$ of $(\gamma_0,\lambda_0)$. Since $(\gamma_0,\lambda_0)$ was arbitrary and $\delta$ only needed to be sufficiently small, this completes the proof.

\end{proof}
\section{Extending Regular Closedness using Perturbations}\label{sec:analytic}
\subsection{Sign-Change Zeros}
We can eliminate the need for property U:
\begin{theorem}\label{thm:N_prime_result}
	$\mdiags{\N'}\subseteq \overline{\interior{\N}}$
\end{theorem}
\begin{proof}
	Let $(\gamma,\lambda)\in \N'$, and let $f\in\B$ with $f(\gamma)=f(\lambda)=0$, where both $\gamma$ and $\lambda$ are sign changes for $f$.
	Since $f$ is a uniformly converging power series, its zeros are of finite order, and because both are sign changes their orders must be odd.
	Without loss of generality, suppose $1,2k+1$ are the orders of $\gamma,\lambda$ respectively.
	We will prove that any zeros of any $f\in\B$ of orders $1,2k+1$ are in $\overline{\interior{\N}}$ by induction on $k$.
	
	For $k=0$, both zeros are simple. Therefore property U is satisfied due to Lemma \ref{thm:simple_is_U}, and we can apply Theorem \ref{thm:calegari_analog} to find arbitrarily close interior points.

	For $k>0$, if $f$ satisfies property U around $\lambda$, the result again follows from Theorem \ref{thm:calegari_analog}. Otherwise, we can apply Lemma \ref{thm:not_U_lower_order} and find a subsequence $m_n$ for which there exists zeros $\lambda_{m_n}\to\lambda$ of odd orders $2k_{m_n}+1<2k+1$. At least one of these orders is achieved infinitely often, say at $f_{m_{n_j}}$, with order $2k'+1<2k+1$. As $\gamma$ is simple, due to Hurwitz theorem, for all sufficiently large $N$ there exists a simple zero $\gamma_N$ of $f_N$ with $\gamma_N\to\gamma$. Therefore we find pairs of zeros $(\gamma_{m_{n_j}},\lambda_{m_{n_j}})$ of $f_{m_{n_j}}$ for all sufficiently large $j$, all of them of orders $1,2k'+1$. Therefore by the induction hypothesis, $(\gamma_{m_{n_j}},\lambda_{m_{n_j}})\in \overline{\interior{\N}}$ for sufficiently large $j$. As $(\gamma_{m_{n_j}},\lambda_{m_{n_j}})\to(\gamma,\lambda)$, this completes the proof.
\end{proof}

\subsection{Outlier Points}

We now use the power series representation of $\N$ to gain more partial results.
\begin{definition}\label{def:outlier_points}
	A parameter $(\gamma,\lambda)\in \mdiags{\N}$ is called an {\it outlier} if $(\gamma,\lambda)\notin \overline{\interior{\N}}$.
	Due to Theorem \ref{thm:N_prime_result}, there are no outlier points in $\N'$.
	We denote the set of outlier points by $\N''$.
\end{definition}

Throughout this subsection, let $(\gamma,\lambda)\in\N''$  and let $f\in\B$ be a power series representation with $f(\gamma)=f(\lambda)=0$. We will incrementally prove restrictions on $f$ and $(\gamma,\lambda)$.

The following lemma would be useful:

\begin{lemma}\label{thm:case_12_tail}
	Let $f$ be a non-constant real analytic function and $\lambda$ an isolated zero of $f$. Let $f_n\rightrightarrows f$ all non-constant real analytic. Then:
\begin{enumerate}
	\item\label{thm:case12_tail_a} If $\lambda$ is a zero which is a sign change for $f$, then for a sufficiently large $n$ there exists $\lambda_n$, a zero which is a sign change for $f_n$, and these satisfy $\lambda_n\to\lambda$.
	\item\label{thm:case12_tail_b} Assume that $\lambda$ is a local maximum of $f$, and there exists a series $\alpha_n$ that converges to $\lambda$, such that $f_n(\alpha_n)>0$ for all sufficiently large $n$. Then for all sufficiently large $n$, there exists $\lambda_n$, a zero of which is a sign change for $f_n$, and also $\lambda_n\to\lambda$.
	\item\label{thm:case12_tail_c}  Assume that  $\lambda$ is a local minimum of $f$, and there exists a series $\alpha_n$ that converges to $\lambda$, such that $f_n(\alpha_n)<0$ for all sufficiently large $n$. Then for all sufficiently large $n$, there exists $\lambda_n$, a zero which is a sign for $f_n$, and also $\lambda_n\to\lambda$.
\end{enumerate}
\end{lemma}
\begin{proof}
\begin{enumerate}
\item 
$\lambda$ is a sign change, hence it is of odd order for $f$, say $2k+1$. Let $\epsilon>0$ be sufficiently small such that $f$ contains no other zeros in $[\lambda-\epsilon,\lambda+\epsilon]$ and that $f$ is analytic in the complex $\epsilon$-neighborhood of $\lambda$. By Hurwitz theorem, for sufficiently large $n$, $f_n$ has $2k+1$ zeros in this neighborhood. Since $f$ has real coefficients, at least one of these zeros is real, denoted $\lambda_n$. Without loss of generality, assume that $\lambda_n$ converges, then by uniform convergence, the limit is a zero of $f$ in $[\lambda-\epsilon,\lambda+\epsilon]$. But the only real zero of $f$ in this interval is $\lambda$ by the choice of $\epsilon$.

\item 
Since $\lambda$ is a local maximum of $f$, there exists some $\delta<\lambda$ such that $f(\delta)<0$ and $f$ is increasing on $[\delta,\lambda]$. Due to uniform convergence, for all sufficiently large $n$, we find that $f_n(\delta)<0$.
We also know that $\delta<\alpha_n$ for sufficiently large $n$. Since $f_n(\alpha_n)>0$, there exists a zero with a sign change $\lambda_n$ of $f_n$ in the interval $[\delta,\alpha_n]$. Without loss of generality, assume that $\lambda_n$ converges, say to $\eta$. Clearly $\eta\in[\delta,\lambda]$ and due to uniform convergence, $f(\eta)=0$. 

In the interval $[\delta,\lambda]$, the only zero of $f$ is $\lambda$ (due to the choice of $\delta$). Therefore $\eta=\lambda$. Consequently, $\lambda_n$ satisfies the conditions of the lemma.

\item Follows from the previous one for $-f$ and $-f_n$
\end{enumerate}
\end{proof}

\begin{corollary}\label{thm:case1_2}
	Let $f,f_n\in\B$ and $f_n\rightrightarrows f$. Assume $\gamma$ and $\lambda$ are zeros of $f$ and $\gamma\neq\lambda$. Assume one of $\gamma$ or $\lambda$ satisfies the conditions of part \ref{thm:case12_tail_a} in Lemma \ref{thm:case_12_tail} and the other satisfies the conditions of either part \ref{thm:case12_tail_b} or part \ref{thm:case12_tail_c} of the same lemma. Then $(\gamma,\lambda)\in\overline{\N'}$. In particular $(\gamma,\lambda)\in\overline{\interior{\N}}$
\end{corollary}
\begin{proof}
	Lemma \ref{thm:case_12_tail} finds the zeros $(\gamma_n,\lambda_n)$, and as we also assume $f_n\in\B$, this implies they are in $\N'$.
\end{proof}

First, we restrict the order of zeros of $f$:

\begin{lemma}\label{thm:eliminate_case}
	\begin{enumerate}
		\item One of the zeros is of order 1 and the other of order $2k$ for some $k\in\NN$.
		\item If $0<\gamma<\lambda$ then  $\gamma$ is of order 1 and $\lambda$ is of order $2k$.
	\end{enumerate}
\end{lemma}
\begin{proof}
	For the first part, due to Lemma  \ref{thm:order_of_zeros}, one zero is of order 1, and in particular a sign change for $f$. Due to Theorem \ref{thm:N_prime_result}, only one of the zeros is a sign change (otherwise the point is not an outlier), hence it is the simple zero. The other zero is a local extremum which imply its order is even.
	
	The second part follows from the second part of Lemma \ref{thm:order_of_zeros}.
\end{proof}

We determine the "tail" of any $f\in\B$ on which $f(\gamma)=f(\lambda)=0$:
\begin{lemma}\label{thm:tail_shape}
	Let $(\gamma,\lambda)\in \N''$ with a power series $f\in\B$ such that $f(\gamma)=f(\lambda)=0$. $f$ has a sign change in exactly one of $\gamma,\lambda$, the other is either a local minimum or local maximum.

	There are four cases, in each of them we determine the tail of $f$:
	\begin{enumerate}
		\item If the zero which is a local maximum is positive, then $f$ ends with a tail of pluses.
		
		\item  If the zero which is a local minimum is positive, then $f$ ends with a tail of minuses.
		
		\item If the zero which is a local maximum is negative, then $f$ ends with a tail where all even coefficients are $1$ and all odd coefficients are $-1$.
		
		\item If the zero which is a local minimum is negative, then $f$ ends with a tail where all even coefficients are $-1$ and all odd coefficients are $1$.
		
	\end{enumerate}
\end{lemma}

\begin{proof}
	The second and fourth cases follow from the first and third (respectively) by replacing $+1$ with $-1$. 
	Without loss of generality, assume the zero which is a local extremum is $\lambda$. 
	In each of the cases, we perturb $f$ infinitely often in a direction for which we can apply Corollary \ref{thm:case1_2}, and contradict the point not being in $\overline{\interior{\N}}$. 
	
	In the first case, assume, for the sake of contradiction, that the tail is not all pluses. Then we can define $g_n(x):=f(x)+x^{k_n}$ for a series $k_n\nearrow\infty$ such that $g_n\in\B$. This is possible exactly because there is a sequence of coefficients which are not all ones. By Lemma \ref{thm:B_compactness}, there is a subsequence $g_{n_j}$ which converges uniformly, but the limit must be $f$, i.e.\  $g_{n_j}\rightrightarrows f$, and $g_{n_j}(\lambda)=\lambda^{k_{n_j}}>0$. Therefore we can apply Corollary \ref{thm:case1_2} and the point is in $\overline{\interior{\N}}$, in contradiction.
	
	In the third case, assume otherwise, then there are two options:
\begin{itemize}
	\item There are infinitely many even coefficients not being $1$. We define $g_n(x)=f(x)+x^{2k_n}\in\B$ for this sequence $k_n\nearrow\infty$.
	\item 
	There are infinitely many odd coefficients not being $-1$. We define $g_n(x)=f(x)-x^{2k_n+1}$ for this sequence $k_n\nearrow\infty$.
\end{itemize}
In both cases we find that $g_n\in\B$ and as $\lambda<0$, also $g_n(\lambda)>0$. Therefore Corollary \ref{thm:case1_2} applies in a similar manner.
\end{proof}
By writing this tail as a geometric series, we obtain that for all four cases, both $\gamma,\lambda$ are zeros of polynomials with coefficients in $\{-2,-1,0,1,2\}$. In particular:
\algebraic
\begin{proof}
	We know the tail of $f$ from Lemma \ref{thm:tail_shape}. In the first case, we can write $$f(x)=p_m(x)+\frac{x^{m+1}}{1-x}$$ for a polynomial $p_m$ of degree $m$ with coefficients in $\{-1,0,1\}$.
	Therefore any zero of $f$ satisfies 
	$$(1-x)p_m(x)+x^{m+1}=0$$
	which is a polynomial with coefficients in $\{-2,-1,0,1,2\}$. The other cases follow similarly.
\end{proof}
This leads to the final (and strongest) restriction:

\isolated

\begin{proof}
	Assume, for the sake of contradiction, that $(\gamma_n,\lambda_n)\to(\gamma,\lambda)$ which is not eventually constant, $(\gamma_n,\lambda_n)\in \N$ and let $f_n\in\B$ for $(\gamma_n,\lambda_n)$. Assume without loss of generality that $\gamma$ is simple and $\lambda$ is a local extremum of some even order $2k$. We will prove that for all sufficiently large $n$,  $f_n$ is of the same form as $f$: 
\begin{itemize}
	\item 	First, we claim $\gamma_n$ is simple and $\lambda_n$ is a local extremum. Indeed, $\gamma_n$ is simple due to Hurwitz theorem on some small neighborhood of $\gamma$. As for $\lambda_n$, similarly, for sufficiently large $n$ there are $2k$ complex zeros of $f_n$ around $\lambda$. For sufficiently large $n$, none of them are sign changes because that would imply $(\gamma,\lambda)\in\overline{\N'}\subseteq\overline{\interior{\N}}$. Therefore all real zeros are of even order, including $\lambda_n$.
	\item 
	
	Second, for sufficiently large $n$, the tail of $f_n$ is of the same shape as the tail of $f$ (as in Lemma \ref{thm:tail_shape}), because otherwise $(\gamma_n,\lambda_n)\in\overline{\N'}$ due to Corollary \ref{thm:case1_2}, and hence their limit would lie in $\overline{\N'}\subseteq\overline{\interior{\N}}$. 
\end{itemize}
	
	Suppose $k_n$ is the largest coefficient of $f_n$ that does not match the form of the tail (e.g., for the first case, the largest non-positive coefficient). We claim $k_n\nearrow\infty$. Indeed, $\gamma_n,\lambda_n$ are zeros of polynomials from a discrete set of coefficients. If the degrees of these polynomials would also be bounded, it would imply there are only finitely many of them. Therefore $(\gamma_n,\lambda_n)$ is eventually constant (as this sequence converges) in contradiction.
	
	Therefore $k_n\nearrow\infty$, and we define $g_n$ as in the proof of Lemma \ref{thm:tail_shape}. For instance, in the first case we define $g_n(x)=f_n(x)+x^{k_n}\in \B$. Clearly $g_n\in\B$.	By Lemma \ref{thm:B_compactness} there is a subsequence for which $f_n,g_n\rightrightarrows g$ for some $g\in\B$.
	Since $(\gamma,\lambda)\notin \N'$, $g$ has one zero with a sign change, and the other is a local extremum (as per Lemma \ref{thm:eliminate_case}). Also, by the construction $g_n(\lambda_n)={\lambda_n}^{k_n}>0$ and the conditions of Corollary \ref{thm:case1_2} are satisfied for $g,g_n$ and therefore $(\gamma,\lambda)\in \overline{\N'}$ in contradiction.
\end{proof}

We conjecture there are no such isolated points, but we cannot prove it at the time of writing.

\section{Further Questions}\label{sec:questions}

Regarding the topology $\N$, the most relevant open question is whether isolated points actually exist. We believe there are none.

\begin{conjecture}
	$\N$ contains no isolated points.
\end{conjecture}

The restrictions we found (the tail and the multiplicity) seem very strict, and the existence of such isolated points would seem like an unlikely coincidence.

In this paper, we used perturbations that change a single coefficient from a power series $f\in\B$. Perhaps we can use more complex perturbations that work only on subsets of $\N$, possibly based on inequalities of the parameters.
This approach could allow us to prove that certain regions of $\N$ contain no isolated points. 
Moreover, in $\OO$ there are no "single coefficient perturbations", and finding a different perturbation method is the primary challenge to adapt the methods of this paper to $\OO$.

Another potential application is to the geometry of $\N$: identifying "cusp corners" (previous work on these can be found in \cite{N_general}, specifically Theorem 2.3 which characterizes many cusp corners, but is not exhaustive). A sketch of the argument is as follows:

We can explicitly find a polygonal region inside the convex hull of the attractor, which is completely disjoint from the attractor itself. 
If a translation vector between cylinder sets falls exactly on the boundary of this region (which is equivalent to an algebraic equation between the parameters), then it is necessarily not in the interior (unless a different translation vector is inside the region, but the normalized translation vectors are far apart, we could eliminate this with some work). This provides explicit formulas for points on the boundary of $\N$. Finding a limit of these boundary points (maybe that minimizes or maximizes some function) would result in a cusp corner.

We can aim to further generalize the method of traps:
A promising candidate is classes of IFSs on the plane that are affine but not homogeneous. That is, we can introduce rotations. For example, the family of IFSs $\{x\mapsto \alpha z x, x\mapsto z x + (1-z) \}$ where $\alpha\in(0,2\pi)$ and $z\in\mathbb{C}$ are parameters. This could work by using a limiting argument on a subsequence of cylinder sets that are "almost" translations of each other. We have not explored this direction in this paper. This family of IFSs, and the associated connectedness locus, were not researched before as far as we know.

\section{Additional Proofs}\label{sec:aux_proofs}

\begin{proof}[Proof of Lemma \ref{thm:B_completeness}]
	Since the functions in $\B$ are all complex analytic on the unit disk, it follows that uniform convergence implies uniform convergence of all derivatives; hence $f^{(k)}_n(0)\xrightarrow{n\to\infty} f^{(k)}(0)$. But $f^{(k)}_n(0)\in \{-k!,0,k!\}$ which is a discrete set. Therefore $f^{(k)}_n(0)$ is eventually constant, and the limit must also be in $\{-k!,0,k!\}$. This clearly implies $f\in\B$.
\end{proof}

\begin{proof}[Proof of Lemma \ref{thm:simple_is_U}]
	There exists a simple zero of $f_n$ for sufficiently large $n$, due to Hurwitz theorem. The zero is real because $f$ is real, and the order is 1.
	Assume, for the sake of contradiction, that there exists critical points $\lambda_n$ of $f_n$ arbitrarily close to $\lambda$. This means $f_n'(\lambda_n)=0$ and $\lambda_n\to\lambda$. But $f'_n$ uniformly converges to $f'$ which would imply $f'(\lambda)=0$, contradicting $\lambda$ being simple.
\end{proof}
\begin{proof}[Proof of Lemma \ref{thm:unqiue_implicit_is_continuous}]
	Let $x_i\to x$ be a converging sequence in $K$. Assume, for the sake of contradiction, that $g(x_i)\not\to g(x)$. Therefore there exists an $\epsilon>0$ and a subsequence $x_{i_j}$ for which $\abs{g(x_{i_j})-y}\geq \epsilon$ for all sufficiently large $j$. Since $K$ is closed and $x\in K$ then $g(x)\in M$ is well defined. $M$ is compact, hence $g(x_{i_j})$ contains a converging subsequence $g(x_{i_{j_k}})\to y$ where $y\in M$. But by continuity
	$$
	0=F(x_{i_{j_k}},g(x_{i_{j_k}}))\to F(x,y)=0,
	$$
	and also $F(x,g(x))=0$. Therefore by the uniqueness, we must conclude $y=g(x)$. This contradicts the choice of $x_{i_j}$.
\end{proof}
\begin{proof}[Proof of Lemma \ref{thm:not_U_lower_order}]
	Extend $f$ to a complex analytic function in a small neighborhood of $\lambda$. Since $f$ is changing sign around $\lambda$, so is $f_n$ for all sufficiently large $n$. Hence there exists a zero which is a sign change for $f_n$, say of order $k_n$. $f'$ has a zero of order $k-1$ at $\lambda$. Due to Hurwitz theorem for $f$, for all sufficiently large $n$, $f_n$ has exactly $k$ zeros (counting with multiplicity) in this neighborhood, and in particular $k_n\leq k$. Similarly, using Hurwitz theorem for $f'$, which has a zero of order $k-1$ at $\lambda$, for sufficiently large $n$, $f_n'$ has exactly $k-1$ zeros.
	
	Assume, for the sake of contradiction, that $k_n=k$ for all sufficiently large $n$. 
	Observe that $\lambda_n$ is a zero of order $k_n-1$ for $f_n'$, therefore there can't be any other zero for infinitely many $n$s, but this implies property U (as there will be no more zeros of $f_n'$) in contradiction.
\end{proof}

\begin{proof}[Proof of Lemma \ref{thm:B_compactness}]
	Given that $$\abs{f_n(x)}\leq\frac{1}{1-x},$$ it follows that
	$$
	\norm{f_n}_{L^1[a,b]}\leq \frac{1}{1-\max\{\abs{a},\abs{b}\}}.
	$$
	By Bolzano-Weirstrass there is a convergent subsequence in $L^1[a,b]$, and by Lemma \ref{thm:B_completeness} this limit is in $\B$.
\end{proof}

\section{Acknowledgments}
I am sincerely thankful to my advisor, Professor Boris Solomyak, for his guidance and mentorship throughout my research. His support and expertise have been invaluable to the completion of this work.

This research is a part of master's thesis at the Bar-Ilan University and was supported in part by the Israel Science Foundation grant 1647/23.
\bibliographystyle{alpha}
\bibliography{paper} 

\begin{thebibliography}{BBBP98}

\bibitem[Ban02]{bandt_orig}
Christoph Bandt.
\newblock On the {Mandelbrot} set for pairs of linear maps.
\newblock {\em Nonlinearity}, 15:1127, 05 2002.

\bibitem[BBBP98]{beaucoup}
Franck Beaucoup, Peter Borwein, David Boyd, and Christopher Pinner.
\newblock Multiple roots of [-1, 1] power series.
\newblock {\em Journal of The London Mathematical Society-second Series - J
  LONDON MATH SOC-SECOND SER}, 57:135--147, 02 1998.

\bibitem[Bou92]{Bousch1992SurQP}
Thierry Bousch.
\newblock Sur quelques probl{\`e}mes de dynamique holomorphe.
\newblock \url{https://www.imo.universite-paris-saclay.fr/~thierry.bousch},
  1992.

\bibitem[Bou93]{bousch1993connexite}
Thierry Bousch.
\newblock Connexit{\'e} locale et par chemins h{\"o}lderiens pour les systemes
  it{\'e}r{\'e}s de fonctions.
\newblock \url{{https://www.imo.universite-paris-saclay.fr/~thierry.bousch}},
  1993.

\bibitem[CKW17]{calegari}
Danny Calegari, Sarah Koch, and Alden Walker.
\newblock Roots, {Schottky} semigroups, and a proof of {Bandt\textquotesingle
  s} conjecture.
\newblock {\em Ergodic Theory and Dynamical Systems}, 37(8):2487--2555, 2017.

\bibitem[EJS24]{espigule2024collinear}
Bernat Espigule, David Juher, and Joan Salda{\~n}a.
\newblock Collinear fractals and bandt's conjecture.
\newblock \url{https://arxiv.org/abs/2411.00160}, 2024.

\bibitem[Hat85]{Hata1985OnTS}
Masayoshi Hata.
\newblock On the structure of self-similar sets.
\newblock {\em Japan Journal of Applied Mathematics}, 2:381--414, 1985.

\bibitem[HI20]{himeki2020regular}
Yutaro Himeki and Yutaka Ishii.
\newblock {$\mathcal{M}_4$} is regular-closed.
\newblock {\em Ergodic Theory and Dynamical Systems}, 40(1):213--220, 2020.

\bibitem[HS17]{hare_sidorov_2017}
Kevin~G. Hare and Nikita Sidorov.
\newblock On a family of self-affine sets: Topology, uniqueness, simultaneous
  expansions.
\newblock {\em Ergodic Theory and Dynamical Systems}, 37(1):193–227, 2017.

\bibitem[Jac96]{jachymski1996continuous}
Jacek Jachymski.
\newblock Continuous dependence of attractors of iterated function systems.
\newblock {\em Journal of mathematical analysis and applications},
  198(1):221--226, 1996.

\bibitem[Shm06]{shmerkin2006overlapping}
Pablo Shmerkin.
\newblock Overlapping self-affine sets.
\newblock {\em Indiana University mathematics journal}, pages 1291--1331, 2006.

\bibitem[Sol15]{N_general}
Boris Solomyak.
\newblock Connectedness locus for pairs of affine maps and zeros of power
  series.
\newblock {\em Journal of Fractal Geometry}, 2(3):281--308, 2015.

\bibitem[SS06]{shmerkin_sol2000}
Pablo Shmerkin and Boris Solomyak.
\newblock Zeros of \{-1, 0, 1\} power series and connectedness loci for
  self-affine sets.
\newblock {\em Experimental Mathematics}, 15:499--511, 01 2006.

\end{thebibliography}
\end{document}